\documentclass{amsart}
\usepackage{hyperref}
\usepackage{amsthm,mathtools}
\usepackage[all]{xy}

\theoremstyle{plain}
\newtheorem{theorem}{Theorem}[section]
\newtheorem{corollary}[theorem]{Corollary}
\newtheorem{definition}[theorem]{Definition}
\newtheorem{example}[theorem]{Example}
\newtheorem{lemma}[theorem]{Lemma}
\newtheorem{problem}[theorem]{Problem}
\newtheorem{proposition}[theorem]{Proposition}
\newtheorem{remark}[theorem]{Remark}

\usepackage[dvipsnames]{xcolor}

\usepackage{amssymb,amsmath,tabularx,graphicx,float,placeins}

\usepackage[labelformat=simple]{subcaption}
\usepackage{tikz}
\usetikzlibrary[calc,intersections,through,backgrounds,arrows,decorations.pathmorphing]

\DeclareMathOperator{\supp}{supp}
\DeclareMathOperator{\dvg}{div}
\DeclareMathOperator{\cost}{cost}

\DeclareMathOperator{\dist}{dist}
\DeclareMathOperator{\Ric}{Ric}
\DeclareMathOperator{\Lip}{Lip}

\newcommand{\SRW}{\operatorname{SRW}}
\newcommand{\sph}{\operatorname{sphere}}
\newcommand{\ball}{\operatorname{ball}}
\newcommand{\abs}{\operatorname{abs}}

\newcommand{\ol}{\overline}
\newcommand{\wt}{\widetilde}

\newcommand{\bbP}{\mathbb{P}}
\newcommand{\bbR}{\mathbb{R}}
\newcommand{\bbZ}{\mathbb{Z}}
\newcommand{\bbT}{\mathbb{T}}
\newcommand{\cG}{\mathcal{G}}
\newcommand{\cK}{\mathcal{K}}
\newcommand{\cX}{\mathcal{X}}
\newcommand{\cY}{\mathcal{Y}}

\newcommand{\fm}{\mathfrak{m}}
\newcommand{\fX}{\mathfrak{X}}
\newcommand{\bfi}{\operatorname{\bf i}}
\newcommand{\bfh}{\operatorname{\bf h}}

\newcommand{\midarrow}{\tikz \draw[-triangle 90] (0,0) -- +(.1,0);}

\usepackage{wasysym}

\definecolor{dark-violet}{RGB}{148, 0, 211}

\begin{document}

\title{Transportation Distance between Probability Measures on the Infinite Regular Tree}
\author{Pakawut Jiradilok and Supanat Kamtue}
\date{\today}

\address{Department of Mathematics, Massachusetts Institute of Technology, Cambridge, Massachusetts}
\email[P.~Jiradilok]{pakawut@mit.edu}

\address{Department of Mathematical Sciences, Durham University, Durham, United Kingdom}
\email[S.~Kamtue]{supanat.kamtue@durham.ac.uk}

\begin{abstract}
In the infinite regular tree $\mathbb{T}_{q+1}$ with $q \in \mathbb{Z}_{\ge 2}$, we consider families $\{\mu_u^n\}$, indexed by vertices $u$ and nonnegative integers (``discrete time steps'') $n$, of probability measures such that $\mu_u^n(v) = \mu_{u'}^n(v')$ if the distances $\dist(u,v)$ and $\dist(u',v')$ are equal. Let $d$ be a positive integer, and let $X$ and $Y$ be two vertices in the tree which are at distance $d$ apart. We compute a formula for the transportation distance $W_1\!\left( \mu_X^n, \mu_Y^n \right)$ in terms of generating functions. In the special case where $\mu_u^n = \mathfrak{m}_u^n$ are measures from simple random walks after $n$ time steps, we establish the linear asymptotic formula $W_1\!\left( \mathfrak{m}_X^n, \mathfrak{m}_Y^n \right) = An + B + o(1)$, as $n \to \infty$, and give the formulas for the coefficients $A$ and $B$ in closed forms. We also obtain linear asymptotic formulas in the cases of spheres and uniform balls as the radii tend to infinity. We show that these six coefficients (two from simple random walks, two from spheres, and two from uniform balls) are related by inequalities.
\end{abstract}


\maketitle

\section{Introduction}
\label{sec:intro}

Optimal transport theory has been used to study Ricci curvature, which is an important geometric object in Riemannian geometry. Ricci curvature is defined on manifolds via the second derivative of the metric tensor, and it captures how fast geodesics deviate from one another. There are two prominent approaches via optimal transport to give generalized notions of Ricci curvature for a larger class of metric measure spaces (with possibly non-smooth structures) in order to describe geometric nature of such spaces.

The first approach due to independent works by Sturm \cite{Sturm1,Sturm2} and by Lott-Villani \cite{LottVill09} (both of which are inspired from the earlier work by Cordero-Erausquin, McCann, and Schmuckenschl{\"a}ger \cite{cms}) is based on the convexity of an entropy functional along geodesics induced by the transport metric with the quadratic cost function. The second approach, which is relevant to this paper, is due to the following observation by Ollivier in \cite{Ollivier}, inspired from the work by von Renesse and Sturm \cite{VRSturm}. In an $n$-dimensional Riemannian manifold, given
two balls centered at points $X$ and $Y$ with the same radius $r>0$ and the distance between their centers equal to $\delta>0$ small enough, the transportation distance between the uniform probability distributions of the two balls can be asymptotically estimated by \[\delta\big(1-\frac{r^2}{2(n+2)}\Ric(v,v) + O(r^3+r^2\delta)\big),\] where the Ricci curvature is calculated at $v$, the unit tangent vector at $X$ in the direction towards $Y$. Intuitively, in the case of positive Ricci curvature, these two balls are closer to each other than their centers are. We recommend a survey by Ollivier \cite{ollivier2011visual}, which provides an excellent visualization of Ricci curvature and this phenomenon. 
Following from the above observation in manifolds, Ollivier defines a generalized notion of Ricci curvature for metric measure spaces, called \emph{coarse Ricci curvature}. This curvature, also known as \emph{Ollivier-Ricci curvature}, has been studied particularly on discrete spaces such as graphs. It has become an active research area in graph theory (see, for example, \cite{Paeng12,JL14,Bhattacharya15,BCLMP18,Benson19,MW19}) as well as in applied fields related to the study of networks (see, for example, \cite{Sandhu15,Sandhu16,Wang16,Sia19}).

In the graph setting, Ollivier-Ricci curvature is defined for a pair of different vertices $X,Y$ to be \[\kappa(X,Y):=1-\frac{W_1(\fm_X,\fm_Y)}{\dist(X,Y)}.\]
Here $\dist(\cdot,\cdot)$ is the graph distance function, given by the number of edges along a shortest path. The function $W_1(\cdot,\cdot)$ is the transportation distance between two probability measures. The measure $\fm_X$ denotes the probability measure obtained from a one-step simple random walk starting at $X$ (and $\fm_Y$ is defined in a similar fashion). 

The motivation of this paper comes from Ollivier's idea in \cite[Examples 4 and 15]{Ollivier} and \cite[Problem C]{Ollsurvey} to study this coarse Ricci curvature at ``a large scale'', i.e., we instead consider measures $\fm_X^n$ and $\fm_Y^n$ obtained by $n$-step random walks from $X$ and from $Y$. We refer interested readers to \cite{Paulin16} for analysis on this multi-step coarse Ricci curvature and \cite{BJL12} for its slight variation. 
In particular, we are interested in the asymptotic behavior of $W_1(\fm_X^n,\fm_Y^n)$ as $n\to\infty$ where $X,Y$ are fixed vertices on the infinite regular tree. The advantage of considering trees is that there is a systematic method to calculate the transportation distance, which we develop in Sections \ref{sec:prelim}, \ref{sec:gentree}, and \ref{sec:radial}. Furthermore, one may use transportation distances in infinite regular trees as upper estimates for transportation distances in Cayley graphs (by viewing trees as their universal covering graphs).

\medskip

We now describe the main results of this paper. Our first main result is Theorem \ref{thm:W-in-gen}, which presents a formula for the transportation distance $W_1 \! \left( \mu_X^n, \mu_Y^n \right)$ between two radially symmetric probability measures $\mu_X^n$ and $\mu_Y^n$ on the infinite regular tree $\mathbb{T}_{q+1}$ in terms of generating functions from the measures. Throughout this paper, $q$ denotes a positive integer at least $2$, and $\mathbb{T}_{q+1}$ denotes the infinite regular tree in which every vertex has degree $q+1$. We think of the superscript $n$ as indicating the ``discrete time step''. At the vertex $X$, we have a sequence of measures $\mu_X^0$, $\mu_X^1$, $\mu_X^2$, $\ldots$, so that at time step $n$, the measure we are considering is $\mu_X^n$. An analogous sequence of measures also exists at the vertex $Y$, and so at time step $n$, we consider the transportation distance between $\mu_X^n$ and $\mu_Y^n$. 

As a consequence of Theorem \ref{thm:W-in-gen}, Corollary \ref{cor:dist-1-formula} gives a shorter formula for the distance in the special case in which the distance $\dist(X,Y)$ between the vertices $X$ and $Y$ is $1$. We reproduce the formula here:
\[
W_1\!\left( \mu_X^n, \mu_Y^n \right) = (2q-2) \cdot [y^n] G_1\!(q,y) + \left( \frac{q+1}{q} \right) \cdot [y^n] G(q,y) - \frac{1}{q} \cdot [y^n] \gamma_0(y).
\]
The notation $[y^n]$ followed by a univariate generating function in $y$ denotes the coefficient of $y^n$ in the generating function. For the precise definitions of the three generating functions $G_1$, $G$, and $\gamma_0$ which appear in the formula above, we refer the readers to Section \ref{sec:Wgenfn}. Here, we briefly describe what they are. For $\ell, n \in \mathbb{Z}_{\ge 0}$, and for any vertices $u$ and $v$ in the graph $\mathbb{T}_{q+1}$ which are at distance $\ell$ apart, and we denote by $g(\ell, n) \ge 0$ the amount of mass at $u$ in the probability distribution $\mu_v^n$ centered at $v$ after $n$ time steps. We let $G(x,y)$ denote the generating function for these masses $g(\ell, n)$ so that
\[
G(x,y) := \sum_{\ell, n \ge 0} g(\ell, n) \cdot x^{\ell} y^n \in \mathbb{R}[\![x,y]\!].
\]
We define $G_1(x,y)$ as $(\partial/\partial x)G(x,y) \in \mathbb{R}[\![x,y]\!]$ and define $\gamma_0(y)$ as $\sum_{n \ge 0} g(0,n) y^n \in \mathbb{R}[\![y]\!]$. Our first main result expresses the transportation distance $W_1\!\left( \mu_X^n, \mu_Y^n \right)$ in terms of these three generating functions. Note that while the formula displayed above is for when $X$ and $Y$ are adjacent vertices in $\mathbb{T}_{q+1}$, Theorem \ref{thm:W-in-gen} presents a longer formula for the general case where the distance $\dist(X,Y)$ can be any positive integer $d$.

Three main examples of families $\{\mu_u^n\}$ of radially symmetric measures we consider in this paper are (i) probability measures from simple random walks, (ii) uniform sphere measures, and (iii) uniform ball measures. These are defined in Section \ref{sec:lin-asymp}. For each of these three cases we study the asymptotic behavior of $W_1\!\left(\mu_X^n, \mu_Y^n\right)$ as $n \to \infty$. Using Theorem \ref{thm:W-in-gen} as a key ingredient in our analysis, we discover that in each of the three cases the transportation distance satisfies $W_1\!\left(\mu_X^n, \mu_Y^n\right) = An + B + o(1)$, as $n \to \infty$, and we manage to compute explicit formulas for $A$ and $B$. We remark that our technique works for any radially symmetric measures with finite support. For example, one can apply our method to study transportation distances between two identical annuli.

In the simple random walk case, we have our second main result, Theorem \ref{thm:SRW-formula}, which is an exact formula of a bivariate generating function. On the infinite regular tree $\mathbb{T}_{q+1}$, consider any two vertices $u$ and $v$ of distance $\dist(u,v) = \ell$ apart, and consider a simple random walk which starts at $u$ with laziness $\alpha$. Suppose that $g(\ell, n)$ denotes the probability that we arrive at $v$ after $n$ steps. We can then form the bivariate generating function
\[
G(x,y) := \sum_{\ell = 0}^{\infty} \sum_{n=0}^{\infty} g(\ell, n) x^{\ell} y^n \in \mathbb{R}[\![x,y]\!].
\]
Our second main result, Theorem \ref{thm:SRW-formula}, presents this generating function in a closed, algebraic form. We note that our formula is similar to one given in Chapter 19 of the book of Woess' \cite{Woess}.

Observe that in the generating function $G(x,y)$ for the simple random walk case, if we specialize $x$ to $0$, we obtain $\gamma(y) := G(0,y) \in \mathbb{R}[\![y]\!]$, which is the generating function for the ``returning probabilities'' of simple random walk. This generating function is well-studied in enumerative combinatorics and probability. We devote Appendix \ref{sec:gamma} to discussing the generating function $\gamma(y)$.

Still in the simple random walk case, our third main result is the linear asymptotic formula for the distance $W_1\!\left( \fm_X^n, \fm_Y^n \right)$ as $n \to \infty$. Recall that the graph we consider here is $\cG = \mathbb{T}_{q+1}$. The two vertices $X$ and $Y$ in $\cG$ are at distance $d \ge 1$ apart, and $\fm_X^n$ and $\fm_Y^n$ are the probability measures from the simple random walks with laziness $\alpha \in [0,1)$ centered at $X$ and $Y$, respectively. For convenience, we use $\delta := \left\lfloor d/2 \right\rfloor$ and $\delta' := \left\lceil d/2 \right\rceil$. Our third main result, Theorem \ref{thm:SRW_An+B}, says that
\[
W_1(\fm_X^n, \fm_Y^n) = A^{\SRW}_{\alpha, d, q} \cdot n + B^{\SRW}_{\alpha, d, q} + o(1),
\]
as $n \to \infty$, where
\[
A^{\SRW}_{\alpha, d, q} = 2(1-\alpha)(q+1 - q^{1-\delta'} - q^{-\delta}) \cdot \frac{q-1}{(q+1)^2},
\]
and
\[
B^{\SRW}_{\alpha, d, q} = d + \frac{2(\delta q^{-\delta} + \delta' q^{1-\delta'})}{q+1} + \frac{2(q^{1-\delta} - q^{1-\delta'})}{(q+1)^2}.
\]
A striking feature of the coefficient $B^{\SRW}_{\alpha, d, q}$ is that it does not depend on $\alpha$, the laziness of the random walks.

We now turn to the case of uniform spheres. The graph $\cG$ is still the infinite regular tree $\mathbb{T}_{q+1}$, and the distance between $X$ and $Y$ is still $d \ge 1$. Let the probability measures $\sigma_X^n$ and $\sigma_Y^n$ be the uniform spheres of radius $n$ centered at $X$ and $Y$ in $\cG$. Our fourth main result, Theorem \ref{thm:sph_An+B}, says that
\[
W_1(\sigma_X^n, \sigma_Y^n) = A^{\sph}_{d, q} \cdot n + B^{\sph}_{d, q} + o(1),
\]
as $n \to \infty$, where
\[
A^{\sph}_{d,q} = \frac{2(q+1-q^{1-\delta'} - q^{-\delta})}{q+1},
\]
and
\[
B^{\sph}_{d,q} = d + \frac{-4q+2(\delta'(q-1)+1)q^{1-\delta'} + 2(\delta(q-1)+q) q^{-\delta}}{q^2-1}.
\]

In the case of uniform balls, we have our fifth main result. When $\beta_X^n$ and $\beta_Y^n$ are the uniform balls of radius $n$ centered at $X$ and $Y$, Theorem \ref{thm:ball_An+B} says that
\[
W_1(\beta_X^n, \beta_Y^n) = A^{\ball}_{d, q} \cdot n + B^{\ball}_{d,q} + o(1),
\]
as $n \to \infty$, where
\[
A^{\ball}_{d,q} = \frac{2(q+1-q^{1-\delta'}-q^{-\delta})}{q+1},
\]
and
\[
B^{\ball}_{d,q} = d + \frac{-6q-2+2(\delta'(q-1)+2)q^{1-\delta'} + 2(\delta(q-1)+q+1)q^{-\delta}}{q^2-1}.
\]
Once again, for precise definitions, we refer the readers to Section \ref{sec:lin-asymp}. We would like to remind the readers that we assume $q\ge 2$ for the above three asymptotic formulas. The analogous questions for $q=1$ are trivial because $\bbT_2=\bbZ$ is the bi-infinite path, and therefore the distances $W_1(\fm_X^n, \fm_Y^n)$, $W_1(\sigma_X^n, \sigma_Y^n)$, and $W_1(\beta_X^n, \beta_Y^n)$ are equal to $\dist(X,Y)$ for any value of $n$.

Rather surprisingly, the six coefficients $A^{\SRW}_{\alpha, d, q}$, $A^{\sph}_{d,q}$, $A^{\ball}_{d,q}$, $B^{\SRW}_{\alpha, d, q}$, $B^{\sph}_{d,q}$, $B^{\ball}_{d,q}$ are related by inequalities. Our sixth main result, Theorem \ref{thm:great-ineq}, says that
\[
0 < A^{\SRW}_{\alpha, d, q} < A^{\sph}_{d,q} = A^{\ball}_{d,q} < 2 \hspace{1 cm} \text{and} \hspace{1 cm} B^{\SRW}_{\alpha, d, q} > B^{\sph}_{d,q} > B^{\ball}_{d,q} \ge \frac{1}{3}
\]
hold for any $\alpha \in [0,1)$, $d \in \mathbb{Z}_{\ge 1}$, and $q \in \mathbb{Z}_{\ge 2}$.

It can be instructive to compare our asymptotic formulas above with the trivial upper bound for the transportation distance. In all three settings above (random walk, sphere, and ball), the probability distributions $\mu_X^n$ and $\mu_Y^n$ are supported on the balls of radius $n$ centered at $X$ and at $Y$, respectively. Let $\mathbf{1}_X$ and $\mathbf{1}_Y$ denote the point masses at $X$ and at $Y$. Then, by the triangle inequality, we have the trivial upper estimate:
\[
W_1\!\left( \mu_X^n, \mu_Y^n \right) \le W_1\!\left( \mu_X^n, \mathbf{1}_X \right) + W_1 \!\left( \mathbf{1}_X, \mathbf{1}_Y \right) + W_1\!\left( \mathbf{1}_Y, \mu_Y^n \right) \le 2n+d.
\]
Our asymptotic formulas above display the six coefficients $A^{\SRW}_{\alpha, d, q}$, $A^{\sph}_{d,q}$, $A^{\ball}_{d,q}$, $B^{\SRW}_{\alpha, d, q}$, $B^{\sph}_{d,q}$, $B^{\ball}_{d,q}$. Let us note what happens to these coefficients as we take the limit $q \to \infty$ (and leave $\alpha$ and $d$ fixed). We see that 
\begin{itemize}
\item $\lim_{q \to \infty} A^{\SRW}_{\alpha, d, q} = 2(1-\alpha)$,
\item $\lim_{q \to \infty} A^{\sph}_{d,q} = \lim_{q \to \infty} A^{\ball}_{d,q} = 2$, and
\item $\lim_{q \to \infty} B^{\SRW}_{\alpha,d,q} = \lim_{q \to \infty} B^{\sph}_{d,q} = \lim_{q \to \infty} B^{\ball}_{d,q} = d$.
\end{itemize}
In the case of simple random walks, we observe that the laziness $\alpha$ still plays a role in the asymptotic formula as $q \to \infty$. For spheres and balls, the asymptotic formula $An+B$ becomes ``closer" to the trivial upper bound of $2n+d$, as $q$ becomes large.

\medskip

Let us now briefly describe techniques we use in deriving our main results in this paper. First, we compute the Wasserstein distance via equivalent reformulations of an optimal transport problem as a minimizing cost of flow and as a maximizing Kantorovich potential. As a result, we obtain a general formula for the Wasserstein distance in Theorem \ref{thm:W1=S1+S2+S3}. Second, we realize the Wasserstein distance as a coefficient of a certain generating function. In special cases we are interested in, including the simple random walk with laziness, the expanding uniform sphere, and the expanding uniform ball, the corresponding generating functions have nice algebraic formulas. With tools from analytic combinatorics, we are able to compute the asymptotic formulas for the coefficients precisely.

We now present the outline of this paper. In Section \ref{sec:prelim}, we review $L^1$-transportation distances and combinatorial flows on graphs. The work in this section is applicable to any locally finite, connected, simple graph. In Section \ref{sec:gentree}, we restrict to the case of (possibly infinite) trees. We discuss methods to compute the transportation distance $W_1$ on trees. In Section \ref{sec:radial}, we restrict further to the case of radially symmetric measures on the infinite regular tree. In Section \ref{sec:S1S2S3}, we compute a formula for $W_1$ in terms of density values of probability measures. Using this formula, we obtain our first main result in Section \ref{sec:Wgenfn}, which expresses $W_1$ in terms of generating functions. 

Section \ref{sec:lin-asymp} is a major section. It contains our second, third, fourth, and fifth main results. We start Section \ref{sec:lin-asymp} with a review of relevant results from complex analysis, and then we prove the main results later in the section. In Section \ref{sec:examples}, we provide illustrations of our main results in the special cases when $d = 1$ and when $d = 2$, where $d$ is the graph theoretical distance between $X$ and $Y$. In Section \ref{sec:ineq}, we prove our final main result: the inequalities between the coefficients from asymptotic formulas. In Section \ref{sec:chi}, we define four interesting graph statistics, compute them in the case of the infinite regular tree, and pose questions of computing them for general graphs. We discuss relationships between our results and coarse Ricci curvature in Section \ref{sec:curvature}.

Appendix \ref{sec:appendix} shows an alternative derivation of Theorem \ref{thm:W1=S1+S2+S3} for a formula of the transportation distance $W_1\!\left( \mu_X^n, \mu_Y^n \right)$. In Appendix \ref{sec:gamma}, we focus on the asymptotic analysis of the generating function $\gamma(y)$ from the analysis of simple random walks in Subsection \ref{subsec:LA-SRW}.

\bigskip

\section{Review of $L^1$-transportation Distances and Combinatorial Flows on Graphs}
\label{sec:prelim}

In this section, we discuss the basics of optimal transport theory on graphs. In particular, we recall the definition of $L^1$-transportation distance, which is also commonly known as $1$-Wasserstein distance function $W_1$.

Throughout this paper, we fix the notation $\cG=(V,E)$ for a graph $\cG$ with the vertex set $V$ and the edge set $E$. Our graph $\cG$ is simple (i.e., it has neither loops nor multiple edges) and connected. For any vertex $v\in V$, let $N(v)$ denote the set of all neighbors of $v$ and let $\deg(v):=|N(v)|$ denote the degree of $v$. The graph $\cG$ is assumed to be locally finite, that is, every vertex has a finite degree. A \emph{measure} $\mu$ on $\cG$ is a nonnegative function $\mu:V\to \mathbb{R}_{\ge 0}$ with a finite support, i.e., $\supp(\mu):=\{v\in V:\ \mu(v)>0 \}$ is a finite set. For convenience, we write $\mu(A):=\sum_{v\in A} \mu(v)$ for any subset $A\subseteq V$, and we write $|\mu|:=\mu(V)=\sum_{v\in V} \mu(v)$.

\begin{definition}
Let $\mu$ and $\nu$ be two measures on $\cG=(V,E)$ such that $|\mu|=|\nu|$. A \emph{transport plan} $\pi$ from $\mu$ to $\nu$ is a function $\pi:V\times V \to \bbR_{\ge 0}$ satisfying the \emph{marginal constraints}:
\begin{equation} \label{eq: marginal_const}
	\sum_{v'\in V}\pi(v,v')=\mu(v) \quad \text{and} \quad
	\sum_{w'\in V}\pi(w',w)=\nu(w).
\end{equation}
For convenience, we also write $\pi(A\times B):=\sum_{v\in A}\sum_{w\in B} \pi(v,w)$ for any subsets $A,B \subseteq V$; in other words, $\pi$ is viewed as a measure on the product space $V\times V$. Then the marginal constraints can be compactly rewritten as $\pi(A\times V) = \mu(A)$ and $\pi(V\times B) = \nu(B)$ for all $A,B \subseteq V$. Moreover, we denote by $\Pi(\mu,\nu)$ the set of all such transport plans $\pi$.

The \emph{total cost} of the plan $\pi$ is given by
\begin{equation} \label{eq: cost_plan}
	\cost(\pi):= \sum_{v\in V}\sum_{w\in V} \dist(v,w)\pi(v,w),
\end{equation}
where $\dist: V\times V\to \bbZ_{\ge 0}$ is the graph distance function.
The \emph{$1$-Wasserstein distance} between $\mu$ and $\nu$ is defined as
\begin{equation} \label{eq: W1_def}
	W_1(\mu,\nu) := \min_{\pi \in \Pi(\mu,\nu)} \cost(\pi). \tag{KP}
\end{equation}
Any transport plan $\pi$ which yields the minimum cost is called \emph{an optimal transport plan}.
\end{definition}

In \eqref{eq: W1_def} above,
we used min instead of inf. The following remark explains why the minimum always exists.
\begin{remark}[Existence of optimal transport plans] \label{rem: pi_exists}

For shortened notation, we write $S_{\mu}:=\supp(\mu)$ and $S_{\nu}:=\supp(\nu)$, and we recall our assumption that they are finite sets. We note from the marginal constraints \eqref{eq: marginal_const} that $\supp(\pi):=\{(v,w)\in V\times V \mid \pi(v,w)>0\}$ is a subset of $S_{\mu}\times S_{\nu}$. 
The Kantorovich problem \eqref{eq: W1_def} can then be viewed as the following finite-dimensional linear program in the standard form:
\begin{align*}
	\textup{minimize} \qquad\,\, &\smashoperator{\sum_{(v,w)\in S_{\mu}\times S_{\nu}}} \dist(v,w)\pi(v,w) \\
	\textup{subject to} \qquad 
	&\smashoperator{\sum_{w\in S_{\nu}}} \pi(v,w)=\mu(v) \qquad \qquad \forall v\in S_{\mu}; \\
	&\smashoperator{\sum_{v\in S_{\mu}}} \pi(v,w)=\nu(w) \qquad \qquad \forall w\in S_{\nu}; \\
	&\pi(v,w) \ge 0 \qquad \qquad \forall (v,w) \in S_{\mu}\times S_{\nu}.
\end{align*} Therefore, a minimizer of this problem always exists.
\end{remark}

Intuitively, a transport plan $\pi$ describes a plan to transport the mass distribution $\mu$ to the mass distribution $\nu$, where $\pi(v,w)$ represents the amount of mass transported from the vertex $v$ to the vertex $w$, and the transportation cost per unit mass is given by the graph distance $\dist(v,w)$. The $1$-Wasserstein distance $W_1(\mu,\nu)$ measures the minimal total transportation cost between $\mu$ and $\nu$. This minimization problem is called the \emph{Kantorovich problem} as it was introduced by Kantorovich in \cite{Kantorovich}.

By viewing \eqref{eq: W1_def} as a linear program, one has an alternative expression of the $1$-Wasserstein distance given by the so-called \emph{Kantorovich dual problem}:
\begin{equation} \label{eq: Kantorovich_dual}
	W_1(\mu,\nu) =\max_{\phi\in \Lip_1(V)} \sum_{v\in V} \phi(v)(\mu(v)-\nu(v)) = \max_{\phi\in \Lip_1(V)}  (\mu-\nu)^\top \phi, \tag{DP}
\end{equation}
where functions in $\bbR^V$ are viewed as vectors, and $f^\top g := \sum_{v\in V} f(v)g(v)$.
Here the maximum is considered among all  $\phi\in\Lip_1(V)$, the class of all functions $\phi:V\to \bbR$ that are $1$-Lipschitz, i.e., $\phi(v)-\phi(w) \le d(v,w)$ for all $v,w\in V$. Any $1$-Lipschitz function $\phi$ which yields the maximum $(\mu-\nu)^\top \phi$ is called an \emph{optimal Kantorovich potential}.

Note that in order to verify the $1$-Lipschitz condition $\phi(v)-\phi(w) \le d(v,w)$ for all pairs of vertices $v,w$, it suffices to check this $1$-Lipschitz condition only for all $\{v,w\}\in E$. Thus the problem \eqref{eq: Kantorovich_dual} can be rewritten as
\begin{equation}
	W_1(\mu,\nu) = \max  \left\{\, (\mu-\nu)^\top \phi \mid \phi:V\to \bbR \text{ such that } \nabla{\phi}(v,w) \le d(v,w) \quad \forall \{v,w\}\in E \,\right\},
\end{equation}
where $\nabla$ is the \emph{discrete gradient} given by $\nabla \phi (v,w) := \phi(v)-\phi(w)$. This new problem (with constraints on the gradient along edges) has a dual problem, known as Beckmann's formula, which minimizes the cost of flows with certain constraints on the divergence. The original work by Beckmann \cite{Beckmann52} is formulated in the continuous setting, and we refer to the discussion in the book by Peyr\'e and Cuturi \cite{Peyre} for the graph setting. Before we formalize this min-cost flow problem, let us provide the definitions of flows and divergence on graphs.

\begin{definition}
Let $\cG=(V,E)$ be a connected and locally finite graph. A \emph{flow} $\psi$ on $\cG$ is a function $\psi: V\times V \to \bbR$ satisfying the following two properties:
	\begin{enumerate}
		\item $\psi(v,w)=-\psi(w,v)$ for all $v,w\in V$, and
		\item $\psi(v,w)=\psi(w,v)=0$ if $\{v,w\}\not\in E$.
	\end{enumerate}
We denote by $\mathfrak{X}(\cG)$ the set of all flows in $\cG$.
The \emph{divergence} is the linear operator $\dvg: \mathfrak{X}(\cG) \to \bbR^V$ defined as
	\begin{equation} \label{eq:dvg_defn}
		(\dvg \psi) (v) := \frac{1}{2}\sum_{w\in V} (\psi (v,w) - \psi (w,v)) = \sum_{w\in V} \psi (v,w),
	\end{equation}
for all $v\in V$ (and the sum can be restricted to those $w\in N(v)$).
	
Moreover, for a given flow $\psi\in \mathfrak{X}(\cG)$, we define $\psi^{\rm abs}:E\to \bbR$ to take the absolute value of $\psi$, that is, for every edge $e=\{v,w\}\in E$, 
\[\psi^{\rm abs}(e):=|\psi(v,w)|=|\psi(w,v)|.\] 
The \emph{(total) cost} of the flow $\psi$ is given by 
	\begin{equation}
		\cost(\psi):=\sum_{e\in E} \psi^{\rm abs}(e).
	\end{equation}
	
A function $\rho: V\to \mathbb{R}$ is called a \emph{zero-sum assignment} if $\supp(\rho)<\infty$ and $\sum_{v\in V} \rho(v)=0$. Given a zero-sum assignment $\rho$, an \emph{admissible flow} for $\rho$ is a flow $\psi\in \mathfrak{X}(\cG)$ which satisfies the \emph{charge-preserving equation} $\dvg \psi = \rho$, or written explicitly as
	\begin{equation} \label{eq: charge-preserv_eq}
		\sum_{w \in N(v)} \psi(v,w) = \rho(v) \qquad \forall v\in V.
	\end{equation}
\end{definition}

Then the Beckmann's problem (or the \emph{min-cost flow} problem) is to minimize the cost of all admissible flows for the assignment $\mu-\nu$, and it can be stated as follows.
\begin{equation} \label{eq: flow_problem}
	W_1(\mu, \nu) = \min_{\dvg \psi=\mu-\nu} \cost(\psi). \tag{FP}
\end{equation}

A flow $\psi\in \mathfrak{X}(\cG)$ can be thought of as a discrete vector field in $\cG$, where $\psi(v,w)$ represents the tangent vector at $v$ in the direction of $v\to w$. For a given $\phi\in \bbR^V$, the gradient $\nabla \phi$ can be regarded as a flow (by a trivial extension $\nabla\phi(v,w):=0$ when $\{v,w\}\not\in E$). The divergence is the adjoint operator of the gradient in the sense that
\[\langle \nabla \phi, \psi \rangle_{\bbR^{V\times V}} = \frac{1}{2} \sum_{\substack{v,w \\ \{v,w\}\in E}} (\phi(v)-\phi(w)) \psi(v,w) =\sum_{v} \left(\phi(v) \sum_{w\in N(v)} \psi(v,w) \right)=\langle \phi, \dvg \psi \rangle_{\bbR^{V}},\]
where $\langle\cdot,\cdot\rangle_{\bbR^{V\times V}}$ and $\langle\cdot,\cdot\rangle_{\bbR^{V}}$ denote the inner products for $\bbR^{V\times V}$ and for $\bbR^V$, respectively.

\begin{remark} \label{rem: pot_flow_exists}
With a similar argument as the one in Remark \ref{rem: pi_exists}, we can consider both optimization problems \eqref{eq: Kantorovich_dual} and \eqref{eq: flow_problem} on a finite subgraph of $\cG$. Thus they are finite-dimensional linear optimization problems, and the existence of an optimal Kantorovich potential $\phi$ and a minimizing flow $\psi$ is guaranteed.
\end{remark}

Another important aspect of the duality $\min\limits_{\pi \in \Pi(\mu,\nu)} \cost(\pi)=\max\limits_{\phi\in \Lip_1(V)}  (\mu-\nu)^\top \phi$ is the following complementary slackness theorem which relates optimal transport plans to optimal Kantorovich potentials.

\begin{theorem}[Complementary slackness] \label{thm: cs}
Let $\cG=(V,E)$ be a locally finite, connected graph and $\mu,\nu$ be measures on $\cG$ such that $|\mu|=|\nu|$. Let $\pi_{opt}$ be an optimal transport plan and $\phi_{opt}$ be an optimal Kantorovich potential. Then for any pair of vertices $v,w\in V$, one has the following implication:
\begin{equation}
	\text{if } \pi_{opt}(v,w)>0, \text{then } \phi_{opt}(v)-\phi_{opt}(w) = \dist(v,w).
\end{equation}
\end{theorem}

\bigskip

\section{Transport Plans, Flows, and Potential Functions on a Tree} 
\label{sec:gentree}

In this section and henceforth, we restrict our graph $\cG=(V,E)$ to be a (possibly infinite) tree. Recall from the previous section that the Wasserstein distance $W_1(\mu,\nu)$ between measures $\mu$ and $\nu$ with $|\mu|=|\nu|$ can be computed via either one of the three optimization problems, namely, min-cost plan \eqref{eq: W1_def}, max potential \eqref{eq: Kantorovich_dual}, or min-cost flow \eqref{eq: flow_problem}:
\begin{equation} \label{eq: W1_threeways}
	W_1(\mu,\nu) = \min_{\pi\in \Pi(\mu,\nu)} \cost(\pi) = \max_{\phi\in \Lip_1(V)}  (\mu-\nu)^\top \phi = \min_{\dvg \psi=\mu-\nu} \cost(\psi).
\end{equation}
The min-cost flow characterization is particularly useful for calculating $W_1(\mu,\nu)$ when our graph $\cG$ is a tree because in such a case there exists a unique flow $\psi$ with $\dvg \psi= \mu-\nu$.

\begin{proposition} \label{prop: unique_flow_tree}
Let $\cG=(V,E)$ be a tree and let $\rho:V\to \mathbb{R}$ be a zero-sum assignment. Then there exists a unique flow $\psi$ with $\dvg \psi=\rho$. More explicitly, for any edge $e=\{x,y\}\in E$, the value of $\psi(x,y)$ is uniquely determined by 
	\begin{equation} \label{eq: flow_tree_defn}
		\psi(x,y) = \sum_{v\in (\cG\backslash \{e\})_x} \rho(v),
	\end{equation}
where $(\cG\backslash \{e\})_x$ denotes the connected component of $\cG\backslash \{e\}$ that contains $x$.
\end{proposition}

Let us remark that in the case of an infinite tree, the assumption that $\rho$ has a finite support guarantees that the flow $\psi$ also has a finite support.

\begin{proof}
Given any edge $e=\{x,y\}\in E$, removing such an edge from the tree $\cG$ divides it into two connected components. One component contains $x$, and the other contains $y$. We denote these components by $(\cG\backslash \{e\})_x$ and $(\cG\backslash \{e\})_y$, respectively.  
	
For a flow $\psi\in \fX(\cG)$ satisfying $\dvg \psi = \rho$, or explicitly \eqref{eq: charge-preserv_eq}, we have $$\sum_{v\in (\cG\backslash \{e\})_x} \rho(v)= \sum_{v\in (\cG\backslash \{e\})_x} \sum_{w \in N(v)} \psi(v,w)= \psi(x,y),$$
where the last equation is because each edge $\{v,w\}$ in $(\cG\backslash \{e\})_x$ is counted twice in the double sum, and $\psi(v,w)+\psi(w,v)=0$ by the skew-symmetric property of the flow $\psi$. Thus the equation \eqref{eq: flow_tree_defn} is proved, and it implies the uniqueness of the flow $\psi$.
\end{proof}

As a consequence of Proposition \ref{prop: unique_flow_tree} and the equation \eqref{eq: W1_threeways}, one can calculate the Wasserstein distance between two measures $\mu$ and $\nu$ such that $|\mu|=|\nu|$ on a tree $\cG=(V,E)$ as the cost of the unique flow $\psi$ that $\dvg \psi=\mu-\nu$:
\begin{equation} \label{eq: W_cal_tree}
	W_1(\mu,\nu) = \sum_{e\in E} \psi^{\rm abs}(e).
\end{equation}

Next we provide an alternative way to compute $W_1(\mu,\nu)$ via the so-called \emph{good potential functions} with respect to the unique flow $\psi$ with $\dvg \psi =\mu-\nu$.

\begin{definition} \label{def: good_potential}
Let $\cG=(V,E)$ be a tree. Consider a zero-sum assignment $\rho:V\to \mathbb{R}$ and the unique flow $\psi$ with $\dvg \psi=\rho$ as given in Proposition \ref{prop: unique_flow_tree}. A function $\Phi:V\to \bbR$ is called a \emph{good potential function} with respect to $\psi$ if for every edge $\{x,y\} \in E$,
	\begin{align} \label{eq: good_potential}
		\Phi(x)-\Phi(y) = 
		\begin{cases}
			1 & \text{ if } \psi(x,y)>0,\\
			-1 & \text{ if } \psi(x,y)<0,
		\end{cases}
	\end{align}
and $\Phi(x)-\Phi(y)\le 1$ if $\psi(x,y)=0$.
\end{definition}
In words, if the flow travels from a vertex $x$ to one of its neighbor $y$, then $x$ must have more potential than $y$. Note that if there is no flow between $x$ and $y$, we do not require that $\Phi(x)=\Phi(y)$ but we only require $\Phi(x)-\Phi(y)\le 1$ to ensure that $\Phi$ is a $1$-Lipschitz function.

Although good potential functions $\Phi$ are not unique, the following lemma asserts that the value of $\rho^\top \Phi$ is independent of the choice of good $\Phi$. 

\begin{lemma}
Let $\cG=(V,E)$ be a tree. Consider a zero-sum assignment $\rho:V\to \mathbb{R}$ and the unique flow $\psi$ with $\dvg \psi=\rho$. Then the value of $\rho^\top \Phi$ is constant for all good potential functions $\Phi$.
\end{lemma}

\begin{proof}
Let us remove all edges $e=\{x,y\}$ in the tree $\cG$ such that $\psi(x,y)=0$. The resulting graph consists of connected components, which we denote by $\cK_1,\cK_2,...$ (and we know that all but finitely many of them are isolated vertices because there are only a finite number of edges with nonzero flow).

For each component $\cK_i$, one can observe from the construction \eqref{eq: good_potential} that the values within $\cK_i$ of any two good potential functions differ only by a constant, that is, for good $\Phi_1$ and $\Phi_2$, there is $c_i\in \bbR$ such that for all $x\in \cK_i$, $\Phi_2(x)-\Phi_1(x)=c_i$. 
Moreover, the total charge within $\cK_i$ is equal to zero: $\sum_{v\in V(\cK_i)} \rho(v)=0$ due to the conservation of charge. Consequently, we deduce that the value of $\rho^\top \Phi:=\sum_{v\in V} \Phi(v)\rho(v)$ is independent of the choice of good $\Phi$, that is,
\[\rho^\top \Phi_2-\rho^\top \Phi_1=\sum_i \sum_{v\in V(\cK_i)} (\Phi_2(v)-\Phi_1(v))\rho(v)=\sum_i \left(c_i\sum_{v\in V(\cK_i)} \rho(v)\right) = 0.\]
\end{proof}

In fact, the following lemma asserts that every good $\Phi$ is an optimal Kantorovich potential, that is, the value of $\rho^\top \Phi$ equals the cost of the flow $\psi$. This gives an alternative formulation for the Wasserstein distance on a tree in terms of $\Phi$.

\begin{lemma} \label{lem: J=flow}
Let $\cG=(V,E)$ be a tree. Consider a zero-sum assignment $\rho:V\to \mathbb{R}$ and the unique flow $\psi$ with $\dvg \psi =\rho$. Then every good potential function $\Phi:V\to \bbR$ defined via \eqref{eq: good_potential} must satisfy
\begin{equation} \label{eq: potential-edge}
\rho^\top \Phi:=\sum_{v\in V} \Phi(v) \rho(v) = \sum_{e\in E} \psi^{\rm abs}(e).
\end{equation}
Consequently, for any given measures $\mu$ and $\nu$ with $|\mu|=|\nu|$,
we have $W_1(\mu,\nu)=(\mu-\nu)^\top \Phi$ for all good potential functions $\Phi$. Thus all these $\Phi$ are optimal Kantorovich potentials.
\end{lemma}

\begin{proof}
Since $\rho^\top \Phi=\sum_{v\in V} \Phi(v)\rho(v)$ is constant for all good potential functions $\Phi$, it suffices to prove \eqref{eq: potential-edge} for at least one good potential function $\Phi:V\to \bbR$. We construct a good potential function explicitly via the following two steps.
\begin{enumerate}
\item[i)] Let $\sigma: V\times E \to \{0,+1,-1\}$ be defined as follows. For any vertex $v$ and any edge $e=\{x,y\}$,
\begin{align} \label{eq: sigma_v_e}
	\sigma(v,e):= \begin{cases}
		0 &\text{ if } \psi(x,y)=0, \\
		+1 &\text{ if } \psi(x,y)>0 \text{ and } v\in (\cG\backslash\{e\})_{x},\\
		&\text{ or } \psi(x,y)<0 \text{ and } v\in (\cG\backslash\{e\})_{y},\\
		-1 &\text{ if } \psi(x,y)>0 \text{ and } v\in (\cG\backslash\{e\})_{y},\\
		&\text{ or } \psi(x,y)<0 \text{ and } v\in (\cG\backslash\{e\})_{x}.\\
	\end{cases}
\end{align}

\item[ii)] Assign $\Phi(v):= \frac{1}{2} \sum_{e\in E} \sigma(v,e)$ for all $v\in V$.
\end{enumerate}

We first check that $\Phi$ is a good potential function. For any edge $\{v,w\}\in E$, we have
$$\Phi(v)-\Phi(w)=\frac{1}{2}\sum_{e\in E} (\sigma(v,e)-\sigma(w,e)) = \frac{1}{2}(\sigma(v,\{v,w\})-\sigma(w,\{v,w\})).$$
We can see from \eqref{eq: sigma_v_e} that it satisfies \eqref{eq: good_potential}. Therefore, $\Phi$ is indeed a good potential function.

For each edge $e=\{x,y\}$, we have from \eqref{eq: flow_tree_defn} that 
$$\psi(x,y)=\sum_{v\in (\cG\backslash \{e\})_x} \rho(v) = - \sum_{v\in (\cG\backslash \{e\})_y} \rho(v) = \frac{1}{2} \sum_{v\in (\cG\backslash \{e\})_x} \rho(v) - \frac{1}{2} \sum_{v\in (\cG\backslash \{e\})_y} \rho(v).$$
Applying \eqref{eq: sigma_v_e}, we obtain for $e=\{x,y\}$,\\
in the case of $\psi(x,y)>0$ that
$$0<\psi(x,y)= \frac{1}{2} \sum_{v\in V} \rho(v) \sigma(v,e),$$
and in the case of $\psi(x,y)<0$ that
$$0>\psi(x,y)= -\frac{1}{2} \sum_{v\in V} \rho(v) \sigma(v,e).$$
Therefore, we always have $$\psi^{\rm abs}(e)=|\psi(x,y)|= \frac{1}{2} \sum_{v\in V} \rho(v) \sigma(v,e),$$ which is also true in the case of $\psi(x,y)=0$.

The cost of the flow $\psi$ can then be calculated as
\begin{align*}
\sum_{e\in E} \psi^{\rm abs}(e) &= \frac{1}{2} \sum_{e\in E} \sum_{v\in V} \rho(v) \sigma(v,e)\\
&= \frac{1}{2} \sum_{v\in V} \rho(v)  \sum_{e\in E} \sigma(v,e) = \sum_{v\in V} \rho(v)\Phi(v),
\end{align*}
which proves the equation \eqref{eq: potential-edge}. In view of \eqref{eq: W_cal_tree}, it follows immediately that $W_1(\mu,\nu)=(\mu-\nu)^\top \Phi$.
\end{proof}

\begin{remark} \label{rem: cs_flow}
Let us consider any pair of vertices $v,w\in V$ such that $\pi_{opt}(v,w)>0$ for some optimal transport plan $\pi_{opt}$.
By the complementary slackness theorem (Theorem \ref{thm: cs}), we know that $\Phi(v)-\Phi(w)=\dist(v,w)$. If the vertices along the path $P_{vw}$ is labeled by $v=v_0, v_1, \ldots, v_\ell=w$ where $\ell:=\dist(v,w)$, then
\[ \ell = \Phi(v)-\Phi(w) = \sum_{i=0}^{\ell-1} \Phi(v_{i})-\Phi(v_{i+1}).\]
Since $\Phi(v_{i})-\Phi(v_{i+1})\in [-1,1]$, we deduce that $\Phi(v_{i})-\Phi(v_{i+1})=1$ for all $i$, which means $\psi(v_{i},v_{i+1})>0$. In other words, the flow $\psi$ travels from $v$ to $w$ whenever there is an optimal transport plan $\pi_{opt}$ such that $\pi_{opt}(v,w)>0$.
\end{remark}

Let us conclude with the following two methods to calculate $W_1(\mu,\nu)$ on a locally finite tree $\cG=(V,E)$, for any given finitely supported measures $\mu$ and $\nu$ such that $|\mu|=|\nu|$.

\, \\ \noindent \textbf{Calculation via the flow:}
\begin{enumerate}
	\item Define the zero-sum assignment $\rho:V\to \bbR$ by $\rho:=\mu-\nu$.
	\item Define the unique flow $\psi$ with $\dvg \psi = \rho$ given by
	$$\psi(x,y) = \sum_{v\in (G\backslash \{e\})_x} \rho(v),$$ for every edge $e=\{x,y\}\in E$
	(see Proposition \ref{prop: unique_flow_tree}).
	\item Compute $W_1(\mu,\nu)$ by the cost of the flow $\psi$, that is,
	$$W_1(\mu,\nu)= \sum_{e\in E} \psi^{\rm abs}(e).$$
\end{enumerate}

\noindent \textbf{Calculation via the potential:} \\
Do steps (1) and (2), but not (3), before doing the following two steps.
\begin{enumerate} \setcounter{enumi}{3}	 	
	\item Define a good potential function $\Phi:V\to \bbR$ which satisfies
	\begin{align*}
		\Phi(x)-\Phi(y) = 
		\begin{cases}
			0 & \text{ if } \psi(x,y)=0,\\
			1 & \text{ if } \psi(x,y)>0,\\
			-1 & \text{ if } \psi(x,y)<0.
		\end{cases}
	\end{align*}
	\item Compute $W_1(\mu,\nu)$ by the cost of the potential function $\Phi$, that is,
	\[W_1(\mu,\nu)= \rho^\top\Phi=\sum_{v\in V} \Phi(v)\rho(v).\]
\end{enumerate}

In the following example, we will demonstrate both methods of calculation.
\begin{example}
Consider a tree $\cG=(V,E)$ given in Figure~\ref{fig: tree_example_1}. The number labeled at each vertex shows the value of the assignment $\rho=\mu-\nu$ (and for non-labeled vertices, their values are $0$). Note that $\rho$ has zero sum: $\sum_{v\in V} \rho(v)=0$.

Now we are going to determine the unique flow $\psi$ with $\dvg \psi =\rho$, and we will then calculate $W_1(\mu,\nu)$ via the flow. We pick an edge in $\cG$, for example, the edge $e=\{y,z\}$. After we remove the edge $e$, the total charge in the component $(\cG\backslash\{e\})_{z}$ is equal to $+1+1+1=+3$ (and the total charge in the other component $(\cG\backslash\{e\})_{y}$ is $-3$). We assign $\psi(z,y)=3$, and label $\psi^{\rm abs}(e)=3$ on this edge with an arrow $z\to y$ to indicate that the potential at $z$ is higher than at $y$. 
	
There could possibly be an edge such that after removing it, the total charge in each of the two components is exactly zero; for example, the edge $\{x,y\}$ satisfies this property. In this case, $\psi(x,y)=0=\psi^{\rm abs}(\{x,y\})$, which indicates the absence of flow between $x$ and $y$. Even in the case that $\cG$ is an infinite tree, since $\mu$, $\nu$ are assumed to have finite supports, there must be only finitely many edges with nonzero flow. We proceed to label all such edges with the size and the direction of the flow as shown in Figure \ref{fig: tree_example_2}. The Wasserstein distance is then given by the sum of the size of the flow $\psi$ on all edges:
\[W_1(\mu,\nu)=1+1+1+1+3+3+1+1=12.\]

Now we are going to calculate $W_1(\mu,\nu)$ via the potential. We fix a potential value at an arbitrary vertex, say, we set the bottom left vertex to have $\Phi(v_0)=0$ as shown in Figure~\ref{fig: tree_example_3}. Next we find potential values for all vertices by tracing along edges starting from $v_0$. The potential value of a succeeding vertex either increases by one (if it is on the upstream), decreases by one (if it is on the downstream), or remains unchanged (if there is no flow from a preceding vertex). The resulting potential is presented in Figure~\ref{fig: tree_example_3}. The Wasserstein distance is the sum of the products between charge and potential at each vertex (and vertices with no charge may be ignored):
\[W_1(\mu,\nu)=4(+1)+3(+1)+2(-2)+0(-1)+1(-2)+3(+1)+4(+1)+4(+1)=12.\]
	
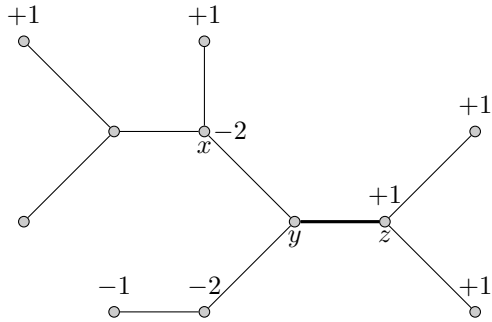
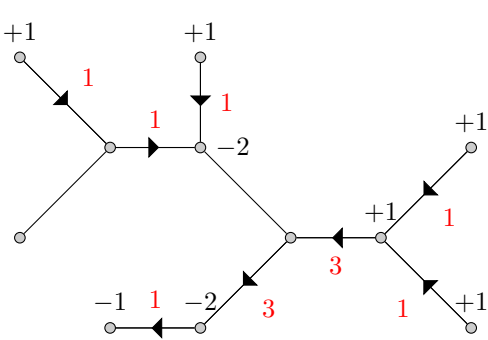
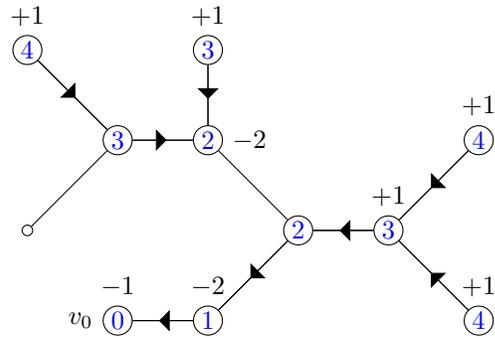
\begin{figure}[h!]
	\begin{subfigure}[t]{0.8\textwidth}
	\centering
	\tikzstyle{every node}=[circle, draw, fill=black!20, inner sep=0pt, minimum width=4pt]
	\begin{tikzpicture}[scale=1.2]
		\draw (0,2) node(v1) [label=above:$+1$] {}
		(0,0) node(v2) {}
		(1,1) node(v3) {}
		(1,-1) node(v4)  [label=above:$-1$] {}
		(2,2) node(v5) [label=above:$+1$] {}
		(2,1) node(v6) [label=right:$-2$] [label=below:$x$] {}
		(2,-1) node(v7) [label=above:$-2$] {}
		(3,0) node(v8) [label=below:$y$] {}
		(4,0) node(v9) [label=above:$+1$] [label=below:$z$] {}
		(5,1) node(v10) [label=above:$+1$] {}
		(5,-1) node(v11) [label=above:$+1$] {};
		
		\draw (v1)--(v3)--(v6)--(v8)--(v9)--(v10); \draw (v2)--(v3); \draw (v4)--(v7)--(v8);  \draw (v5)--(v6); \draw (v9)--(v11); 
		\draw[very thick] (v8)--(v9);
	\end{tikzpicture}
	\caption{A tree $\cG=(V,E)$ has the zero-sum assignment $\rho$ labeled on vertices.}
	\label{fig: tree_example_1}
	\end{subfigure}
	\begin{subfigure}[t]{0.45\textwidth}
	\centering
	\tikzset{vertex/.style={circle, draw, fill=black!20, inner sep=0pt, minimum width=4pt}}
	\begin{tikzpicture}[scale=1.2]
		\draw (0,2) node(v1) [vertex, label=above:$+1$] {}
		(0,0) node(v2) [vertex] {}
		(1,1) node(v3) [vertex] {}
		(1,-1) node(v4)  [vertex, label=above:$-1$] {}
		(2,2) node(v5) [vertex, label=above:$+1$] {}
		(2,1) node(v6) [vertex, label=right:$-2$] {}
		(2,-1) node(v7) [vertex, label=above:$-2$] {}
		(3,0) node(v8) [vertex] {}
		(4,0) node(v9) [vertex, label=above:$+1$] {}
		(5,1) node(v10) [vertex, label=above:$+1$] {}
		(5,-1) node(v11) [vertex, label=above:$+1$] {};
		
		\draw (v1)--(v3)--(v6)--(v8)--(v9)--(v10); \draw (v2)--(v3); \draw (v4)--(v7)--(v8);  \draw (v5)--(v6); \draw (v9)--(v11);
		\begin{scope}[nodes={sloped,allow upside down}]
		\draw (v1) -- node[label=above:\textcolor{red}{$1$}] {\midarrow} (v3);
		\draw (v3) -- node[label=above:\textcolor{red}{$1$}] {\midarrow} (v6);
		\draw (v5) -- node[label=above:\textcolor{red}{$1$}] {\midarrow} (v6);
		\draw (v7) -- node[label=below:\textcolor{red}{$1$}] {\midarrow} (v4);
		\draw (v8) -- node[label=above:\textcolor{red}{$3$}] {\midarrow} (v7);
		\draw (v9) -- node[label=above:\textcolor{red}{$3$}] {\midarrow} (v8);
		\draw (v10) -- node[label=above:\textcolor{red}{$1$}] {\midarrow} (v9);
		\draw (v11) -- node[label=above:\textcolor{red}{$1$}] {\midarrow} (v9);
		\end{scope}
	\end{tikzpicture}
	\caption{\textcolor{red}{Red} numbers and arrows represent the size and direction of the flow $\psi$ on each edge.}
	\label{fig: tree_example_2}
	\end{subfigure}
	~\hspace{0.05\textwidth}
	\begin{subfigure}[t]{0.45\textwidth}
	\centering
	\tikzset{vertex/.style={circle, draw, fill=white, inner sep=1pt, minimum width=4pt}}
	\begin{tikzpicture}[scale=1.2]
		
		\draw (0,2) node(v1) [vertex, label=above:$+1$] {\textcolor{blue}{$4$}}
		(0,0) node(v2) [vertex] {}
		(1,1) node(v3) [vertex] {\textcolor{blue}{$3$}}
		(1,-1) node(v4)  [vertex, label=above:$-1$, label=left:$v_0$] {\textcolor{blue}{$0$}}
		(2,2) node(v5) [vertex, label=above:$+1$] {\textcolor{blue}{$3$}}
		(2,1) node(v6) [vertex, label=right:$-2$] {\textcolor{blue}{$2$}}
		(2,-1) node(v7) [vertex, label=above:$-2$] {\textcolor{blue}{$1$}}
		(3,0) node(v8) [vertex] {\textcolor{blue}{$2$}}
		(4,0) node(v9) [vertex, label=above:$+1$] {\textcolor{blue}{$3$}}
		(5,1) node(v10) [vertex, label=above:$+1$] {\textcolor{blue}{$4$}}
		(5,-1) node(v11) [vertex, label=above:$+1$] {\textcolor{blue}{$4$}};
		
		\draw (v1)--(v3)--(v6)--(v8)--(v9)--(v10); \draw (v2)--(v3); \draw (v4)--(v7)--(v8);  \draw (v5)--(v6); \draw (v9)--(v11);
		\begin{scope}[nodes={sloped,allow upside down}]
			\draw (v1) -- node {\midarrow} (v3);
			\draw (v3) -- node {\midarrow} (v6);
			\draw (v5) -- node {\midarrow} (v6);
			\draw (v7) -- node {\midarrow} (v4);
			\draw (v8) -- node {\midarrow} (v7);
			\draw (v9) -- node {\midarrow} (v8);
			\draw (v10) -- node{\midarrow} (v9);
			\draw (v11) -- node {\midarrow} (v9);
		\end{scope}
	\end{tikzpicture}
	\caption{\textcolor{blue}{Blue} numbers represent the potential $\Phi$ on each vertex.}
	\label{fig: tree_example_3}
\end{subfigure}
\caption{The diagrams demonstrate two methods of calculating the Wasserstein distance on a tree. \subref{fig: tree_example_1} shows
an example of a zero-assignment $\rho$ on a tree, \subref{fig: tree_example_2} shows the calculation via the flow, and \subref{fig: tree_example_3} shows the calculation via the potential.}
\label{fig: tree_example_all}
\end{figure}
\end{example}

\bigskip

\section{Flows and Potential Functions between Two Radially Symmetric Measures on the Infinite Regular Tree}
\label{sec:radial}

In this section, we further restrict our graph to be $\cG=\bbT_{q+1}$, the infinite regular tree in which every vertex has degree exactly $q+1$, with $q\ge 2$. For a fixed pair of vertices $X,Y\in V$ with distance $\dist(X,Y)=d$ apart, we aim to calculate the Wasserstein distance $W_1(\mu_X,\mu_Y)$ between two measures $\mu_X$ and $\mu_Y$, where $\{\mu_u\}_{u\in V}$ is a family of \emph{radially symmetric} measures. For the precise meaning, we fix a sequence of nonnegative real numbers $s=(s(0), s(1),s(2),\ldots)$ with finitely many nonzero terms. For a non-degeneracy, we will also assume that $\sum_{i=0}^{\infty} s(i) >0$.
For each $u\in V$, define a measure $\mu_u$ by
\[
\mu_u(v):= s(\ell) \text{ for all } v \text{ with } \dist(u,v)=\ell. 
\]

In order to describe the unique flow $\psi$ and a good potential $\Phi$ with respect to the assignment $\rho:=\mu_X-\mu_Y$, we will need a good bookkeeping method to refer to the vertices of $\cG$. We describe the bookkeeping as follows.

First, we denote by $P_{XY}$ the unique path from $X$ to $Y$, and we label the vertices along this path by
\[X = Z_0, Z_1, \ldots, Z_d = Y.\]

We define a function $\bfi:V \to \{0,1,...,d\}$ such that for any vertex $v\in V$,
\[ \{Z_{\bfi(v)}\} := P_{vX} \cap P_{vY}\cap P_{XY}.\]
We remark that, for any vertices $a,b,c$ in a tree, the three shortest paths $P_{ab}$, $P_{bc}$, and $P_{ca}$ intersect at exactly one vertex, which is known as the unique \emph{median} of $\triangle abc$.

This vertex $Z_{\bfi(v)}$ is called the \emph{basepoint} of $v$ on the path $P_{XY}$. Furthermore, we define another function $\bfh: V\to \bbZ_{\ge 0}$ as $\bfh(v):= \dist (v,Z_{\bfi(v)})$ to be the \emph{height} of $v$ above its basepoint $Z_{\bfi(v)}$. The vertex set $V$ can then be partitioned into \[V=\bigsqcup_{i=0}^{d} \bigsqcup_{h=0}^{\infty} V_{i,h},\]
where $V_{i,h} :=\{ v\in V: \ \bfi(v)=i, \bfh(v) = h \}$.

The next two lemmas describe the direction of the flow $\psi$ (with $\dvg \psi = \mu_X-\mu_Y$) along each edge of $\cG$.

\begin{lemma} \label{lem: flow_above_ground}
Let $\psi$ be the flow with $\dvg \psi = \mu_X-\mu_Y$. Consider an edge $e=\{v,w\}$ not contained in the path $P_{XY}$, and suppose that $v$ is further away from $P_{XY}$ than $w$ is. Then
\begin{align*}
\psi(v,w) \begin{cases}
\ge 0 &\text{if } \bfi(v)< \frac{d}{2}, \\
=0 &\text{if } \bfi(v)= \frac{d}{2}, \\
\le 0 &\text{if } \bfi(v)> \frac{d}{2},
\end{cases}
\end{align*}
where $Z_{\bfi(v)}$ is the basepoint of $v$ (and also of $w$) on the path $P_{XY}$.
\end{lemma}
In words, on any edge which is not on the path $P_{XY}$, this flow $\psi$ must travel toward the path $P_{XY}$ if the edge is closer to $X$ than to $Y$, and it travels away from $P_{XY}$ if the edge is closer to $Y$ than to $X$.

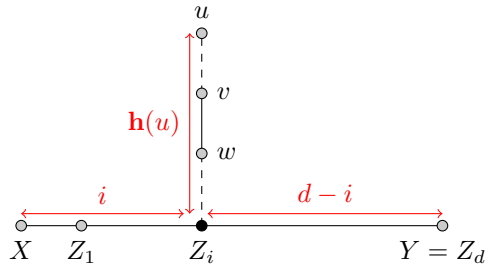
\begin{figure}[h!]
	\centering
	\tikzset{vertex/.style={circle, draw, fill=black!20, inner sep=0pt, minimum width=4pt}}
	\begin{tikzpicture}[scale=0.8]
		\draw (0,0) -- (1,0) -- (3,0) -- (4,0) -- (7,0);
		\draw[dashed] (3,0) -- (3,1.2);
		\draw (3,1.2) -- (3,2.2);
		\draw[dashed] (3,2.2) -- (3,3.2);
		\node at (0,0) [vertex, label={[label distance=0mm]270: ${X}$}] {};
		\node at (1,0) [vertex, label={[label distance=0mm]270: ${Z_1}$}] {};
		\node at (3,0) [vertex, label={[label distance=0mm]270: ${Z_{i}}$}, fill=black] {};
		\node at (7,0) [vertex, label={[label distance=0mm]270: ${Y=Z_d}$}] {};
		\node at (3,1.2) [vertex, label={[label distance=0mm]0: ${w}$}] {};
		\node at (3,2.2) [vertex, label={[label distance=0mm]0: ${v}$}] {};
		\node at (3,3.2) [vertex, label={[label distance=0mm]90: ${u}$}] {};
		\draw[red,<->] (2.8,0.2) -- (2.8,3.2) node [midway, left] {$\bfh(u)$};
		\draw[red,<->] (0,0.2) -- (2.7,0.2) node [midway, above] {$i$};
		\draw[red,<->] (3.1,0.2) -- (7,0.2) node [midway, above] {$d-i$};
	\end{tikzpicture}
	\caption{The diagram illustrates the distances from a vertex $u$ to its basepoint, to $X$, and to $Y$.}
	\label{fig: i+h}
\end{figure}

\begin{proof}
By the equation \eqref{eq: flow_tree_defn}, we have
\begin{align*}
\psi(v,w)= \sum_{u\in (\cG\backslash\{e\})_v} \mu_X(u)-\mu_Y(u),
\end{align*}
where the values of $\mu_X(u)$ and $\mu_Y(u)$ depend on $\dist(X,u)$ and $\dist(Y,u)$, respectively.

Let $\bfi(v)=i$. Observe that all $u\in (\cG\backslash\{e\})_v$ share the same basepoint $Z_i$, and the distances from this basepoint to $X$ and to $Y$ are equal to $i$ and $d-i$, respectively. This gives 
\begin{align*}
\dist(X,u) &= i+\bfh(u),\\
\dist(Y,u) &= d-i+\bfh(u),
\end{align*}
where we note that $\bfh(u)\ge\bfh(v)$ for all $u$; see Figure \ref{fig: i+h} for illustration.

Moreover, given a positive integer $h \ge \bfh(v)$, the number of vertices $u\in (\cG\backslash\{e\})_v$ such that $\bfh(u)=h$ is equal to $q^{h-\bfh(v)}$.	
Therefore, we can compute
\begin{align*}
\psi(v,w)
&=\sum_{u\in (G\backslash\{e\})_v} \mu_X(u)-\mu_Y(u)\\
&=\sum_{u\in (G\backslash\{e\})_v} s(i+\bfh(u))- s(d-i+\bfh(u))\\
&=\sum_{h = \bfh(v)}^\infty q^{h-\bfh(v)} \big(s(i+h) -s(d-i+h)\big).
\end{align*}
If $i=\frac{d}{2}$, then obviously $\psi(v,w)=0$.
Suppose $i<\frac{d}{2}$. Letting $j:=d-2i>0$, we have
\begin{align*}
\psi(v,w)=\sum_{h = \bfh(v)}^{\bfh(v)+j-1}  s(i+h)q^{h-\bfh(v)} +
\sum_{h = \bfh(v)+j}^\infty s(i+h)\big(q^{h-\bfh(v)} - q^{h-\bfh(v)-j} \big) \ge 0.
\end{align*}
Similarly, if $i>\frac{d}{2}$, then $\psi(v,w)<0$ as desired.
\end{proof}

The previous lemma describes the direction of the flow along all edges which are not on the path $P_{XY}$. The following lemma asserts furthermore that the flow along the path $P_{XY}$ simply travels from $X$ to $Y$.

\begin{lemma} \label{lem: half-half}
Let $\psi$ be the flow with $\dvg \psi = \mu_X-\mu_Y$. Then $\psi$ must travel from $X$ to $Y$, that is,
for any edge $e=\{Z_{i},Z_{i+1}\}$ ($0\le i \le d-1$) on the path $P_{XY}$, we have $\psi(Z_{i},Z_{i+1})>0$.
\end{lemma}

\begin{figure}[h!]
	\centering
	\tikzset{vertex/.style={circle, draw, fill=black!20, inner sep=0pt, minimum width=4pt}}
	\begin{tikzpicture}[scale=1.2]
		
		\draw[red!10, fill=red!5] (-3,-2) rectangle (0.5,2);
		\draw[red!20, fill=red!10] (-1.7,1) ellipse (1 and .5);
		
		\draw[blue!10, fill=blue!5] (3.5,-2) rectangle (6.5,2);
		
		\draw (0,0) node{} -- (1,0) node{} -- (3,0) node{} -- (4,0) node{};
		\draw[very thick] (0,0)--(4,0);
		
		\foreach \p in {2,3,4}
		{
			\draw (0,0) node{} -- (-1, 3-\p) node{};
			\foreach \q in {1,0,-1}
			{\draw (-1, 3-\p) node{} -- (-1.5, 3-\p+0.3*\q); }
			\node at (-1,3-\p) [vertex, label={[label distance=0mm]270: ${Z^{\p}}$}] {};
		}
		\foreach \p in {2,3,4}
		{
			\draw (4,0) node{} -- (5, 3-\p) node{};
			\foreach \q in {1,0,-1}
			{\draw (5, 3-\p) node{} -- (5.5, 3-\p+0.3*\q); }
			\node at (5,3-\p) [vertex] {};
		}

		\node at (0,0) [vertex, label={[label distance=0mm]270: ${X}$}] {};
		\node at (1,0) [vertex, label={[label distance=0mm]270: ${Z_1}$}] {};
		\node at (2,0) [label={[label distance=0mm]270: ${\ldots}$}] {};
		\node at (2,0) [label={[label distance=0mm]90: $P_{XY}$}] {};
		\node at (3,0) [vertex, label={[label distance=0mm]270: ${Z_{d-1}}$}] {};
		\node at (4,0) [vertex, label={[label distance=0mm]270: ${Y}$}] {};
		
		\node at (0.5,2) [label={[label distance=0mm]225: $\textcolor{red}{\cX}$}] {};
		\node at (6.5,2) [label={[label distance=0mm]225: $\textcolor{blue}{\cY}$}] {};
		
		\node at (-2.5,1) [label={[label distance=0mm]0: {\large $ U_{Z^2}$}}] {};
		
	\end{tikzpicture}
	\caption{The diagram illustrates the set $\cX$ (and respectively, $\cY$) consisting of all vertices which are closer to $X$ (and respectively, to $Y$) than to all other $Z_{i}$'s.}
	\label{fig:half-half}
\end{figure}
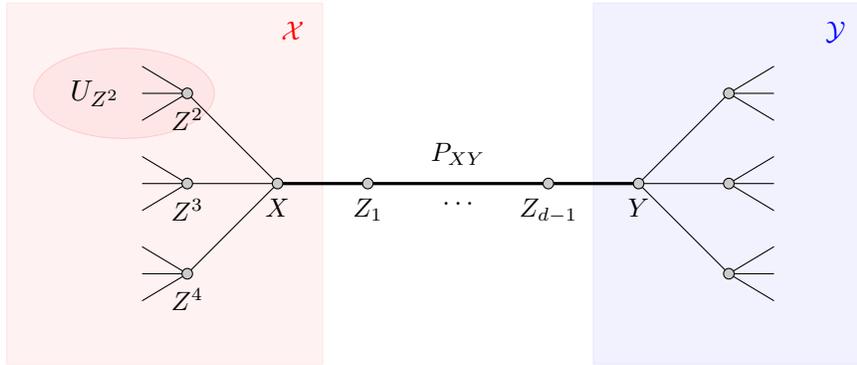

\begin{proof}
Let us denote the set of neighbors of $X$ by $N(X):=\{Z^1,Z^2,...,Z^{q+1}\}$, where $Z^1=Z_1$. We then partition the vertex set by $V = \{X\} \sqcup \bigsqcup\limits_{Z^i \in N(X)} U_{Z^i}$, where each set $U_{Z^i}$ consists of all vertices $v\not=X$ such that the path $P_{vX}$ contains $Z^i$. Moreover, we consider the complement set $\cX:= (U_{Z^1})^{\rm c}$, which consists of all vertices which are closer to $X$ than to $Z^1=Z_1$. By the radial symmetry around the vertex $X$, the measure $\mu_X$ is equally distributed among the $q+1$ sets $U_{Z^i}$, that is, we have $\mu_X(U_{Z^i})=\frac{1}{q+1}(\mu_X(V)-\mu_X(X))$ for all $1\le i \le q+1$. It follows that \[ \mu_X(\cX) = \mu_X(X) + \sum_{i=2}^{q+1}\mu_X(U_{Z^i}) > \frac{1}{2}\mu_X(V)=\frac{1}{2}\pi(V\times V).\]

Similarly, we denote by $\cY$ the set containing all vertices which are closer to $Y$ than to $Z_{d-1}$, and conclude that $\mu_Y(\cY)> \frac{1}{2}\pi(V\times V)$. See Figure \ref{fig:half-half} for the illustration of the sets $\cX$ and $\cY$. For any transport plan $\pi\in \Pi(\mu_X,\mu_Y)$, we know from the Principle of Inclusion-Exclusion that
\begin{align*}
\pi(V\times V) &\ge \pi(\cX\times V) + \pi(V\times \cY) - \pi(\cX\times \cY) \\
&=  \mu_X(\cX) + \mu_Y(\cY) - \pi(\cX\times \cY) \\
&> \frac{1}{2}\pi(V\times V) + \frac{1}{2}\pi(V\times V) - \pi(\cX\times \cY),
\end{align*}
which guarantees $\pi(\cX\times \cY)>0$, that is, there exist a pair of vertices $v_0\in \cX$ and $w_0\in \cY$ such that $\pi(v_0,w_0)>0$. In view of Remark \ref{rem: cs_flow}, if $\pi$ is particularly chosen to be an optimal transport plan, then the flow $\psi$ must travel from $v_0$ to $w_0$. On the other hand, since $v_0\in \cX$ and $w_0\in \cY$, we know that if we travel from $v_0$ to $w_0$ along the unique path $P_{v_0w_0}$, we will visit $X$ before visit $Y$, which means $\psi$ also travels from $X$ to $Y$.
\end{proof}

Now since we know the direction of the flow $\psi$ with $\dvg \psi = \mu_X-\mu_Y$ along each edge of $\cG$, we can construct a good potential function $\Phi: V\to \bbR$ with respect to $\psi$ as follows.

\begin{corollary} \label{cor: construct_good_potential}
Let $\psi$ be the flow with $\dvg \psi = \mu_X-\mu_Y$. Then a function $\Phi: V\to \bbR$ defined as
\begin{align*}
\Phi(v) := \begin{cases}
	\frac{d}{2} - \bfi(v) + \bfh(v) & \text{if } \bfi(v) \le \frac{d}{2}, \\
	\frac{d}{2} - \bfi(v) - \bfh(v) & \text{if } \bfi(v) > \frac{d}{2},
\end{cases}
\end{align*}
is a good potential function with respect to $\psi$.
\end{corollary}

\begin{proof}
One can easily verify that this function $\Phi$ satisfies \eqref{eq: good_potential} where the direction of the flow $\psi$ is described as in Lemmas \ref{lem: flow_above_ground} and \ref{lem: half-half}.
\end{proof}

\bigskip

\section{Distance between Radially Symmetric Distributions}
\label{sec:S1S2S3}
Let $q \ge 2$ and $d \ge 1$ be positive integers. We consider the infinite regular tree $\cG = \mathbb{T}_{q+1}$, in which every vertex has degree exactly $q+1$. In this section, we continue the discussion of radially symmetric measures from Section \ref{sec:radial}. Let us now consider a family $\{\mu_u^n\}_{u \in V(\cG), \, n \in \mathbb{Z}_{\ge 0}}$ of probability distributions (indexed by $u$ and $n$) on $\cG$ with the following property: {\em for every non-negative integer $n$, and for any vertices $u, v, u', v' \in \cG$ such that $\dist(u,v) = \dist(u',v')$, we have $\mu_u^n(v) = \mu_{u'}^n(v')$}. In this and the following sections, we introduce a new function $g: \mathbb{Z}_{\ge 0} \times \mathbb{Z}_{\ge 0} \to \mathbb{R}_{\ge 0}$ given by $g(\ell, n) := \mu_u^n(v)$, for any pair $u, v$ of vertices which are distance $\ell$ apart.

Like in Section \ref{sec:radial}, we let $X$ and $Y$ be two vertices in $\cG$ which are distance $d$ apart. The main goal of this section is to derive Theorem \ref{thm:W1=S1+S2+S3}. The theorem expresses the transportation distance $W_1(\mu_X^n, \mu_Y^n)$ in terms of a summation involving $g(\ell, n)$. As we described in Section \ref{sec:gentree}, there are two ways to compute the distance. We can do this either via the flow, or via the potential. Our approach in this section is via the potential. For interested readers, we also show the calculation via the flow in Appendix \ref{sec:appendix}, which indeed gives the same answer. Our opinion is that for the problem we are dealing with, the potential method is simpler.

To compute the transportation distance between $\mu_X^n$ and $\mu_Y^n$ via the potential, we need to construct a good potential with respect the flow $\psi$ whose divergence is $\mu_X^n - \mu_Y^n$. (We refer the readers back to Definition \ref{def: good_potential} for the precise definition of a ``good potential'' with respect to $\psi$.) Our strategy for this section is to give explicitly a good potential function $\Phi: V(\cG) \to \mathbb{R}$ with respect to $\psi$. From Lemmas \ref{lem: flow_above_ground} and \ref{lem: half-half}, it will be straightforward to check that $\Phi$ to be constructed is indeed good.

Recall that in Section \ref{sec:radial}, we have partitioned the vertex set $V(\cG)$ into
\[
V(\cG) = \bigsqcup_{i=0}^d \bigsqcup_{h=0}^{\infty} V_{i,h}.
\]
We assign the potential $\Phi: V(\cG) \to \mathbb{R}$ as follows. For any vertex $u \in V_{i,h}$, define
\[
\Phi(u) := \begin{cases}
\frac{d}{2} - i + h & \text{if } i \le \frac{d}{2}, \\
\frac{d}{2} - i - h & \text{if } i > \frac{d}{2}.
\end{cases}
\]
Since the potential $\Phi$ is good with respect to $\psi$ (due to Corollary \ref{cor: construct_good_potential}), Lemma \ref{lem: J=flow} then implies
\begin{equation}\label{eq:W1potential}
W_1(\mu_X^n, \mu_Y^n) = \sum_{u \in V(\cG)} \Phi(u) \left( \mu_X^n(u) - \mu_Y^n(u) \right),
\end{equation}
which we can calculate explicitly.

We break the sum on the right hand side of Equation (\ref{eq:W1potential}) into three sums: $S_1 + S_2 + S_3$, where the sum $S_1$ is the sum over the set
\[
\left( \bigcup_{h=0}^{\infty} V_{0,h} \right) \cup \left( \bigcup_{h=0}^{\infty} V_{d,h} \right),
\]
the sum $S_2$ is over the set
\[
\bigcup_{i=1}^{d-1} V_{i,0},
\]
and the sum $S_3$ is over
\[
\bigcup_{i=1}^{d-1} \bigcup_{h=1}^{\infty} V_{i,h}.
\]
Computing each sum directly, we find the following formula.

\begin{theorem}\label{thm:W1=S1+S2+S3}
We have
\[
W_1(\mu^n_X, \mu^n_Y) = S_1 + S_2 + S_3,
\]
where
\[
S_1 = 2 \cdot \sum_{h=0}^{\infty} q^h \left( g(h,n) - g(h+d, n) \right) \left( \frac{d}{2} + h \right),
\]
\[
S_2 = 2 \cdot \sum_{i=1}^{\left\lfloor d/2 \right\rfloor} \left( g(i,n) - g(d-i,n) \right) \left( \frac{d}{2} - i \right),
\]
and
\[
S_3 = 2 \cdot \sum_{i=1}^{\left \lfloor d/2 \right\rfloor} \sum_{h=1}^{\infty} (q-1) q^{h-1} \left( g(i+h,n) - g(d-i+h,n) \right) \left( \frac{d}{2} - i + h \right).
\]
\end{theorem}

Examples of families of radially symmetric measures for which Theorem \ref{thm:W1=S1+S2+S3} is applicable are (i) measures from simple random walks with laziness, (ii) uniform measures on expanding spheres, and (iii) uniform measures on expanding balls. We will give precise definitions and study these examples in great detail in Section \ref{sec:lin-asymp}.

\bigskip

\section{Distance in terms of Generating Functions}
\label{sec:Wgenfn}
In the previous section, we have written the transportation distance $W_1(\mu_X^n, \mu_Y^n)$ in terms of $g(\ell, n)$. In this section, we will write the distance in terms of the $y^n$ coefficient of a certain generating function. Equivalently, we are studying the generating function
\[
\sum_{n = 0}^{\infty} W_1\!\left( \mu_X^n, \mu_Y^n \right) \cdot y^n.
\]
Our goal for this section is to write the distance $W_1\!\left( \mu_X^n, \mu_Y^n \right)$ in terms of the generating function
\[
G(x,y) := \sum_{\ell, n \ge 0} g(\ell, n) x^{\ell} y^n \in \mathbb{R}[\![x,y]\!],
\]
and other related generating functions. This is Theorem \ref{thm:W-in-gen} below. A special case of Theorem \ref{thm:W-in-gen} is an asymptotic formula in Corollary \ref{cor:gamma-small-good-asymp}. The corollary is a key ingredient in deriving explicit formulas in Section \ref{sec:lin-asymp}.

From $G(x,y)$, we construct for each $i = 0, 1, \ldots, d-1$,
\[
\gamma_i(y) := \frac{1}{i!} \frac{\partial^i G(x,y)}{\partial x^i} \Bigg|_{x = 0} = \sum_{n \ge 0} g(i,n) y^n \in \mathbb{R}[\![y]\!].
\]
Let us also define
\[
G_1\!(x,y) := \frac{\partial G(x,y)}{\partial x} \in \mathbb{R}[\![x,y]\!].
\]
We will write $W_1\!\left( \mu_X^n, \mu_Y^n \right)$ in terms of $G$, $G_1$, and $\gamma_i$.

We will use the usual notation for extracting the $y^n$ coefficient of a generating function: if $\mathcal{F}(y) \in \mathbb{R}[\![y]\!]$, then we use $[y^n] \mathcal{F}(y)$ to denote the coefficient in front of $y^n$ of $\mathcal{F}(y)$.

Recall from Theorem \ref{thm:W1=S1+S2+S3} that we have $W_1\!\left( \mu_X^n, \mu_Y^n \right) = S_1 + S_2 + S_3$. The theorem displays the distance $W_1$ in terms of $g(\ell, n)$, while currently we would like the distance $W_1$ in terms of generating functions $G, G_1, \gamma_i$. In the following, we derive the generating-function formulas for $S_1$, $S_2$, $S_3$ separately. From the $g(\ell, n)$-formula for $S_1$ in Theorem \ref{thm:W1=S1+S2+S3}, we obtain the following formula:
\[
S_1 = 2q(1-q^{-d}) \cdot [y^n] G_1\!(q,y) 
+ d(1+q^{-d}) \cdot [y^n] G(q,y) 
- \sum_{i=0}^{d-1} (d-2i) q^{i-d} \cdot [y^n] \gamma_i(y).
\]
Similarly, for $S_2$, we have
\[
S_2 = \sum_{i=1}^{d-1} (d-2i) \cdot [y^n] \gamma_i(y).
\]
The formula for $S_3$ turns out to be more involved, so we will break $S_3$ into three smaller pieces. We write
\[
S_3 = S^{\heartsuit}_3 + S^{\diamondsuit}_3 - S^{\spadesuit}_3,
\]
where
\[
S^{\heartsuit}_3 = \sum_{i=1}^{\left\lfloor d/2 \right\rfloor} \sum_{h=1}^{\infty} (q-1) q^{h-1}(g(i+h,n) - g(d-i+h,n))\cdot(d-2i),
\]
\[
S^{\diamondsuit}_3 = \sum_{i=1}^{\left\lfloor d/2 \right\rfloor} \sum_{h=1}^{\infty} (q-1)q^{h-1}g(i+h,n) \cdot 2h,
\]
and
\[
S^{\spadesuit}_3 = \sum_{i=1}^{\left\lfloor d/2 \right\rfloor} \sum_{h=1}^{\infty} (q-1)q^{h-1}g(d-i+h,n) \cdot 2h.
\]

From the formula of $S^{\heartsuit}_3$, we find
\begin{align*}
S^{\heartsuit}_3 &= \left\{ \left( \frac{d}{q} - \frac{2}{q-1} \right) + \left( \frac{2}{q-1} + d \right) q^{-d} \right\} \cdot [y^n] G(q,y) \\
&\hphantom{=} + \frac{d(q-1)}{q} \cdot [y^n] \gamma_0(y) + \sum_{i=0}^{d-1} \left( \frac{2}{q-1} - d + 2i - \left( \frac{2}{q-1} + d \right) q^{i-d} \right) \cdot [y^n] \gamma_i(y).
\end{align*}

For convenience, let $\delta := \left\lfloor d/2 \right\rfloor$ and $\delta' := \left\lceil d/2 \right\rceil$. Note that $\delta' \ge \delta \ge 0$ are non-negative integers such that $\delta + \delta' = d$.

For $S^{\diamondsuit}_3$, we find
\begin{align*}
S^{\diamondsuit}_3 &= 2\left(1-q^{-\delta}\right) \cdot [y^n] G_1\!(q,y) \\
&\hphantom{=} + \left( - \frac{2}{q-1} + \frac{2}{q} \delta q^{-\delta} + \frac{2}{q-1} q^{-\delta} \right) \cdot [y^n] G(q,y) \\
&\hphantom{=} + \sum_{i=0}^{\delta} \left( \left( \frac{2i}{q} - \frac{2\delta}{q} - \frac{2}{q-1} \right) q^{i-\delta} + \frac{2}{q-1} \right) \cdot [y^n] \gamma_i(y).
\end{align*}

Similarly, for $S^{\spadesuit}_3$, we find
\begin{align*}
S_3^{\spadesuit} &= 2 \left( q^{1-\delta'} - q^{1-d} \right) \cdot [y^n] G_1\!(q,y) \\
&\hphantom{=} + \left\{ - \left( 2 \delta' + \frac{2}{q-1} \right) q^{-\delta'} + \left( 2d + \frac{2}{q-1} \right) q^{-d} \right\} \cdot [y^n] G(q,y) \\
&\hphantom{=} + \sum_{\substack{i \ge \delta + 1 \\ i \le d-1}} \left( \left( 2i - 2d - \frac{2}{q-1} \right) q^{i-d} + \frac{2}{q-1} \right) \cdot [y^n] \gamma_i(y) \\
&\hphantom{=} + \sum_{i = 0}^{\delta} \left( -2(q^{-\delta'} - q^{-d})iq^i + \left( 2 \delta' + \frac{2}{q-1} \right)q^{i-\delta'} - \left(2d+\frac{2}{q-1}\right) q^{i-d} \right) \cdot [y^n] \gamma_i(y).
\end{align*}

Combining the formulas above and using $W_1\!\left( \mu_X^n, \mu_Y^n \right) = S_1 + S_2 + S^{\heartsuit}_3 + S^{\diamondsuit}_3 - S^{\spadesuit}_3$, we obtain the main theorem of this section.

\begin{theorem}\label{thm:W-in-gen}
The formula for the transportation distance $W = W_1\!\left( \mu_X^n, \mu_Y^n \right)$ is
\begin{align*}
W &= \left\{ 2q + 2 - 2q^{1-\delta'} - 2q^{-\delta} \right\} \cdot [y^n] G_1\!(q,y) \\
&\hphantom{=} + \left\{ d + \frac{d}{q} - \frac{4}{q-1} + 2\delta' q^{-\delta'} + \frac{2}{q-1} q^{-\delta'} + 2\delta q^{-1-\delta} + \frac{2}{q-1} q^{-\delta} \right\} \cdot [y^n] G(q,y) \\
&\hphantom{=} - \frac{d}{q} \cdot [y^n] \gamma_0(y) \\
&\hphantom{=} + \sum_{i=0}^{\delta} \left( \frac{4}{q-1} + \left( 2i - 2\delta' - \frac{2}{q-1} \right) q^{i-\delta'} + \left( \frac{2i}{q} - \frac{2\delta}{q} - \frac{2}{q-1} \right) q^{i-\delta} \right) \cdot [y^n] \gamma_i(y).
\end{align*}
\end{theorem}

The theorem above gives a useful corollary in the case where the $\gamma_i$ terms are ``small'' as $n \to \infty$.

\begin{corollary} \label{cor:gamma-small-good-asymp}
Suppose that for every $i = 0, 1, \ldots, d-1$, we have
\[
\left| [y^n] \gamma_i(y) \right| \to 0,
\]
as $n \to \infty$. Then, the transportation distance $W = W_1\!\left( \mu_X^n, \mu_Y^n \right)$ satisfies
\begin{align*}
W &= \left\{ 2q+2-2q^{1-\delta'} - 2q^{-\delta}\right\} \cdot [y^n] G_1\!(q,y) \\
&\hphantom{=} + \left\{ d + \frac{d}{q} - \frac{4}{q-1} + 2\delta' q^{-\delta'} + \frac{2}{q-1} q^{-\delta'} + 2\delta q^{-1-\delta} + \frac{2}{q-1} q^{-\delta} \right\} \cdot [y^n] G(q,y) + o(1),
\end{align*}
as $n \to \infty$.
\end{corollary}

Examples of families of radially symmetric measures where the assumption of Corollary \ref{cor:gamma-small-good-asymp} is satisfied include (i) the family of measures from simple random walks with laziness, (ii) the family of uniform measures of spheres, and (iii) the family of uniform measures of balls. Precise definitions of these families are given in Section \ref{sec:lin-asymp}, where we investigate the three families closely and produce explicit formulas of them.

We remark that Theorem \ref{thm:W-in-gen} simplifies nicely in the case $d = 1$, where $X$ and $Y$ are adjacent vertices in $\cG$. Here, $\delta' = 1$ and $\delta = 0$. The function $\gamma_0$ is the only $\gamma_i$ term, since $d-1 = 0$.

\begin{corollary}\label{cor:dist-1-formula}
When $\dist(X,Y) = 1$, the transportation distance between $\mu_X^n$ and $\mu_Y^n$ is
\[
W_1\!\left( \mu_X^n, \mu_Y^n \right) = (2q-2) \cdot [y^n] G_1\!(q,y) + \left( \frac{q+1}{q} \right) \cdot [y^n] G(q,y) - \frac{1}{q} \cdot [y^n] \gamma_0(y).
\]
\end{corollary}

\bigskip

\section{Explicit Linear Asymptotic Formulas}
\label{sec:lin-asymp}
The main results of this section are Theorems \ref{thm:SRW_An+B}, \ref{thm:sph_An+B}, and \ref{thm:ball_An+B}, which give explicit linear asymptotic formulas for the transportation distance $W_1(\mu_X^n, \mu_Y^n)$ in the cases of simple random walks, spheres, and balls, respectively. We derive the formulas by combining Corollary \ref{cor:gamma-small-good-asymp} with techniques from generating function theory and from analytic combinatorics.

In Subsection \ref{subsec:cauchy}, we develop some tools from analytic combinatorics. The standard references which we recommend are the book of Flajolet and Sedgewick \cite{FS09} and the book of Lang \cite{Lang99}. In Subsection \ref{subsec:LA-SRW}, we treat the case in which $\mu_X^n$ and $\mu_Y^n$ are radially symmetric measures from simple random walks. This case is more complicated than the two subsequent cases in the section, because of the complexity of generating functions involved. Afterwards, we treat the case of spheres in Subsection \ref{subsec:LA-sph}, and then the case of balls in Subsection \ref{subsec:LA-balls}.

We remark that there is a closely related work done by Woess \cite[Chapter 19]{Woess}. While Woess does not focus on computing transportation distances, Woess studies the probability distribution from the lazy random walk on the infinite regular tree and obtains generating function formulas similar to our formulas in Theorem \ref{thm:SRW-formula}. The approach of Woess' is slightly different than our approach in Subsection \ref{subsec:LA-SRW}.

\bigskip

\subsection{Consequences of Cauchy's Coefficient Formula} \label{subsec:cauchy}
We make a short detour to discuss some useful tools from complex analysis. What we will see in this subsection are standard techniques from analytic combinatorics. Our detour will explore merely a tiny part of the subject to develop useful lemmas. We recommend the book of Flajolet and Sedgewick \cite{FS09} to the readers who would like to delve further into the subject of analytic combinatorics. For complex analysis reference, we recommend the book of Lang \cite{Lang99}. In this subsection, we will first recall Cauchy's Coefficient Formula (Theorem \ref{thm:CCF}). The formula extracts a coefficient of a power series in terms of an integral over a simple loop. As a consequence of the formula, we will exhibit an exponential decay behavior of coefficients in Corollary \ref{cor:expo-decay}. The main goal of this subsection is to show Lemmas \ref{l:simple-pole} and \ref{l:double-pole}, which give estimates of coefficients of functions with pole at $1$. These lemmas are useful in our proofs of Propositions \ref{prop:SRW_G1} and \ref{prop:SRW_G}.

\smallskip

We start by recalling {\em Cauchy's Coefficient Formula}.

\begin{theorem}[Cauchy's Coefficient Formula, cf. e.g. \cite{FS09} or \cite{Lang99}]\label{thm:CCF}
Let $U \subseteq \mathbb{C}$ be a simply connected, open domain containing $0 \in \mathbb{C}$. Let $f: U \to \mathbb{C}$ be a holomorphic function on $U$ with the following series expansion at $0$:
\[
f(z) = a_0 + a_1 z + a_2 z^2 + \cdots.
\]
Suppose that $\gamma$ is a simple counterclockwise closed loop (with winding number $+1$) around $0$ inside the domain $U$. Then,
\[
a_n = \frac{1}{2\pi i} \int_{\gamma} \frac{f(z)}{z^{n+1}} \, dz.
\]
\end{theorem}

We obtain bounds on the magnitude of coefficients of holomorphic functions as a consequence of Cauchy's formula. The following corollary says that if a function is holomorphic on an open disk of radius larger than $1$ centered at $0$, then the magnitudes of coefficients of $f$ (as a series expanded around $0 \in \mathbb{C}$) exhibit an exponential decay. This is known as {\em Cauchy's inequality} or {\em Cauchy's bound}.

\begin{corollary}[Cauchy's inequality]\label{cor:expo-decay}
Let $0 < \varepsilon' < \varepsilon$ be real numbers. Let $D \subseteq \mathbb{C}$ denote the open disk of radius $1+\varepsilon$ centered at $0$. Suppose that $f:D \to \mathbb{C}$ is holomorphic with series expansion
\[
f(z) = a_0 + a_1 z + a_2 z^2 + \cdots.
\]
Then, $|a_n| = O \!\left( (1+\varepsilon')^{-n} \right)$, as $n \to \infty$. In particular, the sequence $\{a_n\}$ converges to $0$.
\end{corollary}
\begin{proof}
Let $r := 1 + \varepsilon'$. Let $\gamma$ be the simple closed loop $\{r e^{i\theta}: 0 \le \theta \le 2\pi\}$, oriented counterclockwise. Cauchy's Coefficient Formula says
\[
a_n = \frac{1}{2\pi i} \int_{\gamma} \frac{f(z)}{z^{n+1}} \, dz.
\]
By switching to polar coordinates, we find
\[
a_n = \frac{1}{2\pi i} \int_0^{2\pi} \frac{f(re^{i\theta})}{r^n e^{ni\theta}} \cdot i \, d\theta.
\]
This implies
\[
|a_n| \le \frac{1}{2\pi} \int_0^{2\pi} |f(re^{i\theta})| \cdot r^{-n} \, d\theta \le r^{-n} \sup_{z \in \gamma} |f(z)|.
\]
Thus, $|a_n| = O(r^{-n})$, as desired.
\end{proof}

Next, we prove two lemmas which give estimates on coefficients of certain functions. We will use these lemmas to prove Propositions \ref{prop:SRW_G1} and \ref{prop:SRW_G}.

\begin{lemma}\label{l:simple-pole}
Let $\varepsilon > 0$. Let $B$ be the open disk of radius $1+\varepsilon$ centered at $0$. Suppose that the function $f: B \setminus \{1\} \to \mathbb{C}$ is holomorphic with the following series expansion around $0$:
\[
f(z) = a_0 + a_1 z + a_2 z^2 + \cdots.
\]
Suppose also that the function $(1-z)f(z)$ can be extended analytically to a holomorphic function $g: B \to \mathbb{C}$. Then,
\[
a_n = g(1) + o(1),
\]
as $n \to \infty$.
\end{lemma}
\begin{proof}
The function
\[
f(z) - \frac{g(1)}{1-z} = \frac{g(z) - g(1)}{1-z}, \tag{$\ast$}
\]
originally defined on the punctured disk $B \setminus \{1\}$, can be analytically extended to a holomorphic function on the whole disk $B$, because $g$ is complex differentiable on $B$. Write this function in ($\ast$) as the following series
\[
b_0 + b_1 z + b_2 z^2 + \cdots
\]
around $0$. By Corollary \ref{cor:expo-decay}, we have that $b_n \to 0$ as $n \to \infty$. Note that the function $\frac{g(1)}{1-z}$ has the expansion
\[
g(1) + g(1) z + g(1) z^2 + \cdots
\]
around $0$. Hence, the relation ($\ast$) shows that $a_n = g(1) + b_n = g(1) + o(1)$, as $n \to \infty$.
\end{proof}

It is not difficult to extend Lemma \ref{l:simple-pole} to higher orders of the pole at $1$. We show the case of double pole (which will be useful in the proof of Proposition \ref{prop:SRW_G1}) in the following lemma. The proof of the following lemma goes in a similar fashion as the one in the previous lemma.

\begin{lemma}\label{l:double-pole}
Let $\varepsilon > 0$. Let $B$ be the open disk of radius $1+\varepsilon$ centered at $0$. Suppose that the function $f: B \setminus \{1\} \to \mathbb{C}$ is holomorphic with the following series expansion around $0$:
\[
f(z) = a_0 + a_1 z + a_2 z^2 + \cdots.
\]
Suppose also that the function $(1-z)^2 f(z)$ can be extended analytically to a holomorphic function $g: B \to \mathbb{C}$. Then,
\[
a_n = g(1) \cdot n + \left( g(1) - g'(1) \right) + o(1),
\]
as $n \to \infty$.
\end{lemma}
\begin{proof}
The proof is analogous to the one in the previous lemma. Instead of the function in ($\ast$), here note that the function
\[
f(z) + \frac{g'(1)}{1-z} - \frac{g(1)}{(1-z)^2} = \frac{\left( \frac{g(z)-g(1)}{z-1} \right) - g'(1)}{z-1},
\]
can be extended analytically to the whole disk $B$.
\end{proof}

\bigskip

\subsection{Linear Asymptotics for Simple Random Walks with Laziness} \label{subsec:LA-SRW}
In this subsection, fix a real number $\alpha \in [0,1)$ and positive integers $d \in \mathbb{Z}_{\ge 1}$, $q \in \mathbb{Z}_{\ge 2}$. Our graph $\cG$ is the infinite regular tree $\mathbb{T}_{q+1}$. The parameter $\alpha$ is referred to as the {\em laziness} of the random walk. We consider the special case in which our family $\{\mu_u^n\}_{u \in V(\cG), n \in \mathbb{Z}_{\ge 0}}$ of probability distributions is from simple random walks on $\cG$ with laziness $\alpha$. More precisely, for each vertex $u$ in $\cG$, consider the simple random walk $p_0, p_1, p_2, \ldots$ where each $p_i$ is a random vertex of $\cG$ given by
\begin{itemize}
\item $p_0 = u$,
\item for each $i \ge 0$, we have $p_{i+1} = p_i$ with probability $\alpha$, and
\item for each $i \ge 0$, and for any neighbor $v$ of $p_i$, we have $p_{i+1} = v$ with probability $\frac{1-\alpha}{q+1}$.
\end{itemize}
In this subsection, we consider the distributions $\mu_u^n$ to be $\mu_u^n = \fm_u^n$, where
\[
\fm_u^n(v) = \bbP\!\left( p_n = v \right),
\]
for every vertex $u, v \in V(\cG)$. That is, $\fm_u^n$ is the probability mass function of the simple random walk which starts at the vertex $u$ at Step $0$ and which has laziness $\alpha$.

The goal of this subsection is to use Corollary \ref{cor:gamma-small-good-asymp} to derive a precise asymptotic formula for $W_1 \! \left( \fm_X^n, \fm_Y^n \right)$, where $X$ and $Y$ are two vertices which are at distance $d$ apart, in the form
\[
W_1 \! \left( \fm_X^n, \fm_Y^n \right) = A_{\alpha, d, q}^{\text{SRW}} \cdot n + B_{\alpha, d, q}^{\text{SRW}} + o(1),
\]
as $n \to \infty$. Here, $A_{\alpha, d, q}^{\text{SRW}}$ and $B_{\alpha, d, q}^{\text{SRW}}$ are real constants depending only on $\alpha, d, q$.

Let us now describe the strategy we take for this subsection. We first set up the linear recurrences for $g(\ell, n)$ in this case in Proposition \ref{prop:SRW-recurrence}. The linear recurrences are solved in Theorem \ref{thm:SRW-formula}, in which we discover the closed forms of relevant generating functions. From the closed forms, we derive the asymptotic formula for the transportation distance in Theorem \ref{thm:SRW_An+B}.

Recall the notation $g(\ell, n)$ from Section \ref{sec:S1S2S3}. In the case of simple random walks, we have the following proposition.
\begin{proposition}\label{prop:SRW-recurrence}
For simple random walks, the numbers $\{g(\ell, n)\}_{\ell, n \ge 0}$ satisfy
\begin{itemize}
\item $g(\ell, n) = 0$ if $\ell > n$,
\item $g(0,0) = 1$,
\item for every $n \ge 0$, we have
\[
g(0, n+1) = \alpha \cdot g(0,n) + (1-\alpha) \cdot g(1,n),
\]
and
\item for every $n \ge 0$ and every $\ell \ge 1$, we have
\[
g(\ell, n+1) = \alpha \cdot g(\ell, n) + q \cdot \frac{1-\alpha}{q+1} \cdot g(\ell+1, n) + \frac{1-\alpha}{q+1} \cdot g(\ell-1, n).
\]
\end{itemize}
\end{proposition}

Recall the notations $G, G_1, \gamma_i$ from Section \ref{sec:Wgenfn}. For convenience, we write $\gamma := \gamma_0 \in \mathbb{R}[\![y]\!]$. From Proposition \ref{prop:SRW-recurrence}, we obtain the functional equation
\begin{align*}
qx &= \left( (q+1)(x-\alpha xy) - q(1-\alpha)y - (1-\alpha)x^2y \right) \cdot G(x,y) \tag{\eighthnote} \\
&\hphantom{=} + \left( -x + \alpha xy + q(1-\alpha) y \right) \cdot \gamma(y).
\end{align*}

Since we find the following idea interesting, let us now describe how we will solve Equation (\eighthnote) before we actually start solving it. We will prove Lemma \ref{lemma:unique-G-gamma}, which says that the solution to Equation (\eighthnote) is unique. With uniqueness, it suffices to simply give an example of a pair of formal series $(G,\gamma)$ which satisfies the equation. Then, the example is the desired solution. We think this is peculiarly interesting, as it seems like we are in a situation where we have ``one equation with two unknowns,'' but the solution $(G,\gamma)$ is nevertheless unique. Note that we do {\em not} need to assume the relation $\gamma(y) = G(0,y)$ between $G$ and $\gamma$ in Lemma \ref{lemma:unique-G-gamma}. In fact, Equation (\eighthnote) {\em implies} that $\gamma(y) = G(0,y)$.

\begin{lemma}\label{lemma:unique-G-gamma}
Suppose that the real number $\alpha \in [0,1)$ and the integer $q \ge 2$ are given. There exist a unique pair of formal power series $\wt{G}(x,y) \in \mathbb{R}[\![x,y]\!]$ and $\wt{\gamma}(y) \in \mathbb{R}[\![y]\!]$ which satisfy
\begin{align*}
qx &= \left( (q+1)(x-\alpha xy) - q(1-\alpha)y - (1-\alpha)x^2y \right) \cdot \wt{G}(x,y) \tag{$\wt{\eighthnote}$} \\
&\hphantom{=} + \left( -x + \alpha xy + q(1-\alpha) y \right) \cdot \wt{\gamma}(y).
\end{align*}
\end{lemma}
\begin{proof}
The existence is clear from our construction earlier. We will show uniqueness. Assuming $\wt{G}$ and $\wt{\gamma}$ satisfy the functional equation above, we will describe how to recover all the coefficients of $\wt{G}$ and $\wt{\gamma}$ uniquely.

First, plugging $x = 0$ into ($\wt{\eighthnote}$) yields
\[
0 = - q(1-\alpha) y \cdot \wt{G}(0,y) + q(1-\alpha) y \cdot \wt{\gamma}(y).
\]
Thus we obtain $\wt{\gamma}(y) = \wt{G}(0,y)$. It suffices to recover $\wt{G}(x,y)$, as the power series $\wt{\gamma}(y)$ can be obtained from $\wt{G}(x,y)$.

We write
\[
\wt{G}(x,y) := \sum_{i, j \ge 0} \wt{g}_{i,j} x^i y^j \in \mathbb{R}[\![x,y]\!].
\]
To recover the coefficients $\wt{g}_{i,j}$ of $\wt{G}$, it suffices to establish the recurrence relations analogous to the ones we have found in Proposition \ref{prop:SRW-recurrence}. Namely, we claim that
\begin{itemize}
\item[(i)] $\wt{g}_{i,j} = 0$ if $i>j$,
\item[(ii)] $\wt{g}_{0,0} = 1$,
\item[(iii)] for every $j \ge 0$, we have
\[
\wt{g}_{0,j+1} = \alpha \cdot \wt{g}_{0,j} + (1-\alpha) \cdot \wt{g}_{1,j},
\]
and
\item[(iv)] for every $j \ge 0$ and every $i \ge 1$, we have
\[
\wt{g}_{i,j+1} = \alpha \cdot \wt{g}_{i,j} + q \cdot \frac{1-\alpha}{q+1} \cdot \wt{g}_{i+1,j} + \frac{1-\alpha}{q+1} \cdot \wt{g}_{i-1,j}.
\]
\end{itemize}
We will obtain the four items in the following order: (ii), (iii), (iv), (i).

\smallskip

\textbf{(ii).} Consider the coefficient of $x$ of both sides of Equation ($\wt{\eighthnote}$). We find
\[
q = (q+1) \cdot \wt{g}_{0,0} - \wt{g}_{0,0}.
\]
This implies $\wt{g}_{0,0} = 1$.

\smallskip

\textbf{(iii).} Taking $\frac{\partial}{\partial x}$ of both sides of Equation ($\wt{\eighthnote}$) and then plugging in $x = 0$, we find
\[
q = \left((q+1)(1-\alpha y)\right) \cdot \wt{G}(0,y) - q(1-\alpha)y \left( \frac{\partial \wt{G}(x,y)}{\partial x} \Bigg|_{x = 0} \right) + (-1+\alpha y) \cdot \wt{\gamma}(y).
\]
Using that
\[
\frac{\partial \wt{G}(x,y)}{\partial x} \Bigg|_{x = 0} = \sum_{j \ge 0} \wt{g}_{1,j} y^j,
\]
and recalling that
\[
\wt{\gamma}(y) = \wt{G}(0,y) = \sum_{j \ge 0} \wt{g}_{0,j} y^j,
\]
we obtain
\[
\sum_{j \ge 0} \wt{g}_{0,j+1} y^{j+1} = \alpha \sum_{j \ge 0} \wt{g}_{0,j} y^{j+1} + (1-\alpha) \sum_{j \ge 0} \wt{g}_{1,j} y^{j+1},
\]
which implies $\wt{g}_{0,j+1} = \alpha \cdot \wt{g}_{0,j} + (1-\alpha) \cdot \wt{g}_{1,j}$, for all $j \ge 0$, as desired.

\smallskip

\textbf{(iv).} Consider any arbitrary integer $k \ge 2$. By collecting the terms whose $x$-degree is exactly $k$ in ($\wt{\eighthnote}$), we obtain
\[
0 = (q+1)(x-\alpha xy) \sum_{j \ge 0} \wt{g}_{k-1,j} x^{k-1} y^j - q(1-\alpha)y \sum_{j \ge 0} \wt{g}_{k,j} x^k y^j - (1-\alpha) x^2 y \sum_{j \ge 0} \wt{g}_{k-2,j} x^{k-2} y^j.
\]
This gives
\[
\wt{g}_{k-1,0} + \sum_{j \ge 0} \wt{g}_{k-1,j+1} y^{j+1} = \alpha \sum_{j \ge 0} \wt{g}_{k-1,j} y^{j+1} + q \frac{1-\alpha}{q+1} \sum_{j \ge 0} \wt{g}_{k,j} y^{j+1} + \frac{1-\alpha}{q+1} \sum_{j \ge 0} \wt{g}_{k-2,j} y^{j+1}.
\]
Therefore, for any $k \ge 2$ and for any $j \ge 0$, we have
\[
\wt{g}_{k-1,j+1} = \alpha \wt{g}_{k-1,j} + q \frac{1-\alpha}{q+1} \wt{g}_{k,j} + \frac{1-\alpha}{q+1} \wt{g}_{k-2,j},
\]
which is the desired recurrence relation. Furthermore, we also obtain $\wt{g}_{k-1,0} = 0$, for all $k \ge 2$.

\smallskip

\textbf{(i).} In the last part of the previous step, we observed that $\wt{g}_{i,0} = 0$, for all $i > 0$. To show that $\wt{g}_{i,j} = 0$, for all $i > j$, we proceed by induction on $j \ge 0$. The base case of $j = 0$ is completed in the previous step. The inductive step is obtained immediately from the recurrence in (iii) we established earlier.

We have finished the proof.
\end{proof}

In the case of simple random walks, we have the formulas for $G(x,y)$ and $\gamma(y)$ in the following theorem. Note that in the special case of $\alpha = \frac{1}{2}$, Woess \cite[Chapter 19]{Woess} also discovers this formula.

\begin{theorem}\label{thm:SRW-formula}
For simple random walks, we have the following formulas:
\begin{align*}
\gamma(y) &= q \left( \left( \frac{q-1}{2} \right) (1-\alpha y) + \sqrt{\Delta} \right)^{-1}, \\
G(x,y) &= \frac{\left( \frac{q+1}{2} \right) (1-\alpha y) + \sqrt{\Delta}}{\left( \frac{q+1}{2} \right) (1-\alpha y) - (1-\alpha)xy + \sqrt{\Delta}} \cdot q \left( \left( \frac{q-1}{2} \right) (1-\alpha y) + \sqrt{\Delta} \right)^{-1},
\end{align*}
where
\[
\Delta = \left( \frac{q+1}{2} \right)^2 (1 - \alpha y)^2 - q(1-\alpha)^2 y^2.
\]
\end{theorem}
\begin{proof}
By Lemma \ref{lemma:unique-G-gamma}, it suffices to show that the described generating functions $G(x,y)$ and $\gamma(y)$ satisfy Equation (\eighthnote). This task can be done by plugging in directly.
\end{proof}

By Corollary \ref{cor:gamma-small-good-asymp}, in order to find the linear asymptotic formula for $W_1\!\left( \fm_X^n, \fm_Y^n \right)$, it suffices to find those for $[y^n]G_1\!(q,y)$ and $[y^n]G(q,y)$. This task is straightforward, albeit rather tedious. We will describe the details of computation briefly.

Let's write

\begin{align*}
& \frac{q}{\gamma} = \phi_1(y), \\
& \frac{G}{\gamma} = \frac{\phi_2(y)}{\phi_3(x,y)},
\end{align*}
where
\begin{align*}
\phi_1(y) &= \left( \frac{q-1}{2} \right)(1-\alpha y) + \sqrt{\Delta} \\
\phi_2(y) &= \left( \frac{q+1}{2} \right)(1-\alpha y) + \sqrt{\Delta} \\
\phi_3(x,y) &= \left( \frac{q+1}{2} \right)(1-\alpha y) - (1-\alpha)xy + \sqrt{\Delta}.
\end{align*}

Since $G_1 = \frac{\partial G}{\partial x}$, we obtain
\[
G_1\!(x,y) = \frac{(1-\alpha)y \cdot G(x,y)}{\phi_3(x,y)}.
\]

Let's also define
\[
\ol{\phi}_3(x,y) = \left( \frac{q+1}{2} \right)(1-\alpha y) - (1-\alpha)xy - \sqrt{\Delta}.
\]
Note that we have the following identity
\[
\phi_3(q,y) \cdot \ol{\phi}_3(q,y) = (1-\alpha)q(q+1)y(y-1).
\]

Define the formal power series $H_1\!(y) := G_1\!(q,y) \cdot (1-y)^2 \in \mathbb{R}[\![y]\!]$. Note that we can express
\[
H_1\!(y) = \frac{1}{(1-\alpha)q(q+1)^2} \cdot \frac{\phi_2(y) \cdot \ol{\phi}_3(q,y)^2}{y \cdot \phi_1(y)}.
\]

It is routine to show that there exists $\varepsilon > 0$ for which $H_1\!(y)$ is a holomorphic function on the open disk of radius $1+\varepsilon$ centered at $0$. Therefore, by Lemma \ref{l:double-pole}, we can write
\[
[y^n]G_1\!(q,y) = H_1\!(1) \cdot n + \left(H_1\!(1) - H'_1\!(1) \right) + o(1),
\]
as $n \to \infty$.

Direct computations show that the formulas for $H_1\!(1)$ and $H'_1\!(1)$ are quite nice:
\[
H_1\!(1) = \frac{(1-\alpha)(q-1)}{(q+1)^2}
\]
and
\[
H'_1\!(1) = \frac{(1-\alpha)(q-1)}{(q+1)^2} - \frac{2q}{(q-1)(q+1)^2}.
\]
Thus, we have proved the following proposition.

\begin{proposition}\label{prop:SRW_G1}
For simple random walks, we have
\[
[y^n] G_1\!(q,y) = \frac{(1-\alpha)(q-1)}{(q+1)^2} \cdot n + \frac{2q}{(q-1)(q+1)^2} + o(1),
\]
as $n \to \infty$.
\end{proposition}

Next, we compute the asymptotic formula for $[y^n]G(q,y)$. While $G_1(q,y)$ has a double pole at $1$, the meromorphic function $G(q,y)$ only has a simple pole at $1$. Define the formal power series $H(y) := (1-y) \cdot G(q,y) \in \mathbb{R}[\![y]\!]$. We have
\[
H(y) = - \frac{1}{(1-\alpha)(q+1)y} \cdot \frac{\phi_2(y) \cdot \ol{\phi}_3(q,y)}{\phi_1(y)}.
\]
Direct computations give
\[
H(1) = \frac{q}{q+1}.
\]
Thus, by applying Lemma \ref{l:simple-pole}, we obtain the following proposition.

\begin{proposition}\label{prop:SRW_G}
For simple random walks, we have
\[
[y^n] G(q,y) = \frac{q}{q+1} + o(1),
\]
as $n \to \infty$.
\end{proposition}

Combining Propositions \ref{prop:SRW_G1} and \ref{prop:SRW_G}, together with Corollary \ref{cor:gamma-small-good-asymp}, we obtain the following main result.
\begin{theorem}\label{thm:SRW_An+B}
For simple random walks, we have
\[
W_1(\fm_X^n, \fm_Y^n) = A^{\SRW}_{\alpha, d, q} \cdot n + B^{\SRW}_{\alpha, d, q} + o(1),
\]
as $n \to \infty$, where
\[
A^{\SRW}_{\alpha, d, q} = 2(1-\alpha)(q+1 - q^{1-\delta'} - q^{-\delta}) \cdot \frac{q-1}{(q+1)^2},
\]
and
\[
B^{\SRW}_{\alpha, d, q} = d + \frac{2(\delta q^{-\delta} + \delta' q^{1-\delta'})}{q+1} + \frac{2(q^{1-\delta} - q^{1-\delta'})}{(q+1)^2}.
\]
Recall that $\delta=\left\lfloor d/2 \right\rfloor$ and $\delta'=\left\lceil d/2 \right\rceil$.
\end{theorem}

Curiously, we note that the coefficient $B^{\SRW}_{\alpha, d, q}$ does not depend on the parameter $\alpha$. This phenomenon where the coefficient of the second-order term does not depend on some other parameter seems ubiquitous in algebra and combinatorics. One elementary, yet elegant, example is that of Faulhaber's formula (also known as Bernoulli's formula). For any positive integer $p$, it is rather well-known that the sum $1^p + 2^p + \cdots + n^p$ is asymptotically $\frac{1}{p+1} \cdot n^{p+1}$. Equivalently, $\sum_{i=1}^n (i/n)^p \sim \frac{1}{p+1} \cdot n$, as $n \to \infty$. What is perhaps less well-known is the second coefficient. It turns out that $\sum_{i=1}^n (i/n)^p = \frac{1}{p+1} \cdot n + \frac{1}{2} + o(1)$, as $n \to \infty$. The coefficient $1/2$ is indeed independent of the parameter $p$.

We also note that Theorem \ref{thm:SRW_An+B} is particularly nice when $d$ is an even positive integer.

\begin{corollary}
Suppose $d = \dist(X,Y)$ is even. Then, for simple random walks,
\[
W_1(\fm_X^n, \fm_Y^n) = 2(1-\alpha)(1-q^{-d/2}) \frac{q-1}{q+1} \cdot n + d(1+q^{-d/2}) + o(1),
\]
as $n \to \infty$.
\end{corollary}

We have obtained the desired linear asymptotic formula for the simple random walk case. In the next subsection, we turn to the sphere case.

We remark that our analysis above computes the asymptotic formulas for $[y^n]G_1\!(q,y)$ and $[y^n]G(q,y)$, but we did not need the formula for $[y^n]\gamma(y)$ in order to obtain Theorem \ref{thm:SRW_An+B}. It turns out that the asymptotics for $\gamma(y)$ is also nice, even though its derivation is rather involved. For interested readers, we present our investigation of $[y^n]\gamma(y)$ in Appendix \ref{sec:gamma}.

\bigskip

\subsection{Linear Asymptotics for Spheres} \label{subsec:LA-sph}
A {\em sphere} of radius $r \in \mathbb{Z}_{\ge 0}$ centered at a vertex $u$ in $\cG$ is the set of all vertices $v$ whose distances from $u$ are exactly $r$. The number of vertices in a sphere of radius $r$ is $(q+1)q^{r-1}$ if $r \ge 1$, and is $1$ if $r = 0$. In this subsection, we consider the case when $\{\mu_u^n\}_{u \in V(G), n \in \mathbb{Z}_{\ge 0}}$ are uniform measures on spheres. More precisely, we consider $\mu_u^n = \sigma_u^n$, where $\sigma_u^n$ is given by
\begin{itemize}
\item for any vertex $u$, we have $\sigma_u^0(u) = 1$,
\item for any vertices $u, v$ with $u \neq v$, we have $\sigma_u^0(v) = 0$,
\item for any $r \in \mathbb{Z}_{\ge 1}$, and for any vertices $u, v$ with $\dist(u,v) = r \ge 1$, we have $\sigma_u^r(v) = (q+1)^{-1}q^{1-r}$, and
\item for any $r \in \mathbb{Z}_{\ge 1}$, and for any vertices $u, v$ with $\dist(u,v) \neq r$, we have $\sigma_u^r(v) = 0$.
\end{itemize}
Observe that the distribution $\sigma_u^n$ is the probability mass function of the discrete uniform distribution on the sphere of radius $n$ centered at $u$.

In this subsection, our goal is to use Corollary \ref{cor:gamma-small-good-asymp} to obtain the asymptotic formula for $W_1 \! \left( \sigma_X^n, \sigma_Y^n \right)$ of the form
\[
W_1 \! \left( \sigma_X^n, \sigma_Y^n \right) = A_{d, q}^{\sph} \cdot n + B_{d, q}^{\sph} + o(1),
\]
as $n \to \infty$.

Computations are easier in the case of spheres than in the case of simple random walks in the previous subsection. For spheres, we do not need to solve bivariate linear recurrences to obtain the closed form of generating functions. We can write down the formula for $G(x,y)$ and $G_1(x,y)$ immediately.

Note that for spheres,
\[
g(\ell, n) = \begin{cases}
1 & \text{ if } \ell = n = 0, \\
(q+1)^{-1}q^{1-n} & \text{ if } \ell = n \ge 1, \\
0 & \text{ otherwise.}
\end{cases}
\]
This implies
\[
G(x,y) = \frac{q^2+q-xy}{(q+1)(q-xy)}
\]
and
\[
G_1\!(x,y) = \frac{q^2y}{(q+1)(q-xy)^2}.
\]
Therefore,
\[
[y^n] G(q,y) = \begin{cases}
1 & \text{ if } n = 0, \\
\frac{q}{q+1} & \text{ if } n \ge 1.
\end{cases}
\]
We also have
\[
[y^n] G_1\!(q,y) = \frac{n}{q+1},
\]
for all $n \ge 0$.

Using Corollary \ref{cor:gamma-small-good-asymp} once again, we have proved the following.
\begin{theorem}\label{thm:sph_An+B}
For spheres, we have
\[
W_1(\sigma_X^n, \sigma_Y^n) = A^{\sph}_{d, q} \cdot n + B^{\sph}_{d, q} + o(1),
\]
as $n \to \infty$, where
\[
A^{\sph}_{d,q} = \frac{2(q+1-q^{1-\delta'} - q^{-\delta})}{q+1},
\]
and
\[
B^{\sph}_{d,q} = d + \frac{-4q+2(\delta'(q-1)+1)q^{1-\delta'} + 2(\delta(q-1)+q) q^{-\delta}}{q^2-1}.
\]
\end{theorem}

In fact, observe that we actually have the {\em exact} formula for the sphere case:
\[
W = A^{\sph}_{d, q} \cdot n + B^{\sph}_{d, q},
\]
for all $n$ sufficiently large. This is because the error term $o(1)$ comes from the $\gamma_i$ terms. For spheres, these $\gamma_i$ are eventually zero when $n$ is sufficiently large.

We turn to the case of uniform balls next.

\bigskip

\subsection{Linear Asymptotics for Balls} \label{subsec:LA-balls}
A {\em ball} of radius $r \in \mathbb{Z}_{\ge 0}$ centered at vertex $u$ in $\cG$ is the set of all vertices $v$ such that $\dist(u,v) \le r$. The number of vertices in a ball of radius $r$ is $\frac{q^{r+1}+q^r-2}{q-1}$, for all $r \ge 0$. This subsection considers the case when $\{\mu_u^n\}_{u \in V(\cG), n \in \mathbb{Z}_{\ge 0}}$ is a family of uniform measures on balls. More precisely, we let $\mu_u^n = \beta_u^n$, where
\begin{itemize}
\item for any $r \in \mathbb{Z}_{\ge 0}$, and for any vertices $u, v$ such that $\dist(u,v) \le r$, we have
\[
\beta_u^r(v) = \frac{q-1}{q^{r+1}+q^r-2},
\]
and
\item for any $r \in \mathbb{Z}_{\ge 0}$, and for any vertices $u, v$ such that $\dist(u,v) > r$, we have $\beta_u^r(v) = 0$.
\end{itemize}

Like the sphere case, we do not need to solve linear recurrences. Note that
\[
g(\ell, n) = \frac{q-1}{q^{n+1}+q^n - 2},
\]
for all $0 \le \ell \le n$, and $g(\ell, n) = 0$ otherwise. This gives
\[
G(x,y) = \sum_{n=0}^{\infty} \frac{q-1}{q^{n+1}+q^n-2} \left( \frac{x^{n+1}-1}{x-1} \right) y^n,
\]
and
\[
G_1\!(x,y) = \sum_{n=0}^{\infty} \frac{q-1}{q^{n+1}+q^n-2} \frac{nx^{n+1} - (n+1)x^n+1}{(x-1)^2} y^n.
\]
For $[y^n]G(q,y)$ and $[y^n]G_1\!(q,y)$, we have the following asymptotic formulas:
\[
[y^n]G(q,y) = \frac{q}{q+1} + o(1),
\]
and
\[
[y^n]G_1\!(q,y) = \frac{n}{q+1} - \frac{1}{q^2-1} + o(1),
\]
as $n \to \infty$.

Using Corollary \ref{cor:gamma-small-good-asymp}, we obtain the following theorem. 
\begin{theorem}\label{thm:ball_An+B}
For uniform measures on balls, we have
\[
W_1(\beta_X^n, \beta_Y^n) = A^{\ball}_{d, q} \cdot n + B^{\ball}_{d,q} + o(1),
\]
as $n \to \infty$, where
\[
A^{\ball}_{d,q} = \frac{2(q+1-q^{1-\delta'}-q^{-\delta})}{q+1},
\]
and
\[
B^{\ball}_{d,q} = d + \frac{-6q-2+2(\delta'(q-1)+2)q^{1-\delta'} + 2(\delta(q-1)+q+1)q^{-\delta}}{q^2-1}.
\]
\end{theorem}

It is interesting to compare the coefficients of the sphere case and the ball case. We observe that $A^{\sph}_{d, q}$ and $A^{\ball}_{d, q}$ share the same formula, while $B^{\sph}_{d,q}$ and $B^{\ball}_{d,q}$ appear to be different. If one considers only the main term of the linear asymptotic formula ($W \sim An$), one would not see the difference between the sphere case and the ball case. One needs to go to the second term ($W = An+B + o(1)$) in order to be able to tell the asymptotics of the two cases apart.

In fact, we know that $B^{\sph}_{d,q}$ is always strictly greater than $B^{\ball}_{d,q}$. This relationship, along with others, will be proved in Section \ref{sec:ineq}.

\bigskip

\section{Applications to Short-Distance Cases}
\label{sec:examples}
We have done hard work in previous sections to derive the general linear asymptotic formulas in the cases of simple random walks, spheres, and balls. In this section, we give examples of special cases when $d = 1, 2$. As one might expect in such small cases, the formulas below are nice and short. Furthermore, the special cases when $d = 1$ and when $d = 2$ seem to appear frequently in graph theory. We decide to present these formulas here so that they become convenient for the readers' future uses.

The case $d = 1$ corresponds to when $X$ and $Y$ are adjacent vertices in the infinite regular tree. This will be considered in Subsection \ref{subsec:ex-d1}. The case $d = 2$ corresponds to when $X$ and $Y$ have a unique common neighbor. We consider this in Subsection \ref{subsec:ex-d2}.

\subsection{Examples when $d = 1$} \label{subsec:ex-d1}
In this subsection, we fix $d = 1$, while we still allow $\alpha \in [0,1)$ and $q \in \mathbb{Z}_{\ge 2}$ to be arbitrary. In this case, in our infinite regular tree $\cG = \mathbb{T}_{q+1}$, the two vertices $X$ and $Y$ are adjacent.

First, we consider the case of simple random walks. Let $\fm_X^n$ and $\fm_Y^n$ be as defined in Subsection \ref{subsec:LA-SRW}. Theorem \ref{thm:SRW_An+B} specializes to

\begin{proposition}
Let $X$ and $Y$ be adjacent vertices in $\cG = \mathbb{T}_{q+1}$. Then,
\[
W_1\!\left( \fm_X^n, \fm_Y^n \right) = 2(1-\alpha) \cdot \frac{(q-1)^2}{(q+1)^2} \cdot n + \frac{q^2+6q+1}{(q+1)^2} + o(1),
\]
as $n \to \infty$.
\end{proposition}

Second, we consider the case of spheres. Let $\sigma_X^n$ and $\sigma_Y^n$ be as defined in Subsection \ref{subsec:LA-sph}. Theorem \ref{thm:sph_An+B} specializes to

\begin{proposition}
Let $X$ and $Y$ be adjacent vertices in $\cG = \mathbb{T}_{q+1}$. Then,
\[
W_1\!\left( \sigma_X^n, \sigma_Y^n \right) = 2 \cdot \frac{q-1}{q+1} \cdot n + 1 + o(1),
\]
as $n \to \infty$.
\end{proposition}
One noteworthy point about the sphere case is that the formula is in fact {\em exact}. We actually have that $W_1\!\left( \sigma_X^n, \sigma_Y^n \right) = 2 \cdot \frac{q-1}{q+1} \cdot n + 1$ holds for all non-negative integers $n$.

Third, we consider the case of balls. Let $\beta_X^n$ and $\beta_Y^n$ be as defined in Subsection \ref{subsec:LA-balls}. Theorem \ref{thm:ball_An+B} specializes to

\begin{proposition}
Let $X$ and $Y$ be adjacent vertices in $\cG = \mathbb{T}_{q+1}$. Then,
\[
W_1\!\left( \beta_X^n, \beta_Y^n \right) = \left( \frac{q-1}{q+1} \right) (2n+1) + o(1),
\]
as $n \to \infty$.
\end{proposition}

We remark that in this case of adjacent $X$ and $Y$, similar to the case of the sphere above, it is easy to find an {\em exact} formula for the transportation distance:
\[
W_1\!\left( \beta_X^n, \beta_Y^n \right) = \frac{q-1}{q+1-2q^{-n}}(2n+1),
\]
for all non-negative integers $n$.

\bigskip

\subsection{Examples when $d = 2$} \label{subsec:ex-d2}
In this subsection, we fix $d = 2$. As before, we allow $\alpha \in [0,1)$ and $q \in \mathbb{Z}_{\ge 2}$ to be arbitrary.

First, we consider the case of simple random walks. Let $\fm_X^n$ and $\fm_Y^n$ be as defined in Subsection \ref{subsec:LA-SRW}. Theorem \ref{thm:SRW_An+B} specializes to

\begin{proposition}
Let $X$ and $Y$ be vertices in $\cG = \mathbb{T}_{q+1}$ with $\dist(X,Y) = 2$. Then,
\[
W_1\!\left(\fm_X^n, \fm_Y^n\right) = 2(1-\alpha) \cdot \frac{(q-1)^2}{q(q+1)} \cdot n + \frac{2(q+1)}{q} + o(1),
\]
as $n \to \infty$.
\end{proposition}

Second, we consider the case of spheres. Let $\sigma_X^n$ and $\sigma_Y^n$ be as defined in Subsection \ref{subsec:LA-sph}. Theorem \ref{thm:sph_An+B} specializes to

\begin{proposition}
Let $X$ and $Y$ be vertices in $\cG = \mathbb{T}_{q+1}$ with $\dist(X,Y) = 2$. Then,
\[
W_1\!\left( \sigma_X^n, \sigma_Y^n \right) = 2 \cdot \frac{q-1}{q} \cdot n + 2 \cdot \frac{q^2+1}{q(q+1)} + o(1),
\]
as $n \to \infty$.
\end{proposition}

Third, we consider the case of balls. Let $\beta_X^n$ and $\beta_Y^n$ be as defined in Subsection \ref{subsec:LA-balls}. Theorem \ref{thm:ball_An+B} specializes to

\begin{proposition}
Let $X$ and $Y$ be vertices in $\cG = \mathbb{T}_{q+1}$ with $\dist(X,Y) = 2$. Then,
\[
W_1\!\left( \beta_X^n, \beta_Y^n \right) = 2 \cdot \frac{q-1}{q} \cdot n + 2 \cdot \frac{q-1}{q+1} + o(1),
\]
as $n \to \infty$.
\end{proposition}

\bigskip

\section{Inequalities between the Coefficients}
\label{sec:ineq}
Our main goal in this section is to prove Theorem \ref{thm:great-ineq}. In Section \ref{sec:lin-asymp}, we have derived the explicit formulas for the coefficients $A^{\SRW}_{\alpha, d, q}$, $A^{\sph}_{d,q}$, $A^{\ball}_{d,q}$, $B^{\SRW}_{\alpha, d, q}$, $B^{\sph}_{d,q}$, $B^{\ball}_{d,q}$. Rather surprisingly, there are relationships between these coefficients, some of which are easier to see than others. In Theorem \ref{thm:great-ineq}, we show that all these six coefficients are {\em always strictly positive} under our running assumption: $\alpha \in [0,1)$, $d \ge 1$, and $q \ge 2$. Furthermore, they satisfy certain ordering properties which hold for every $\alpha, d, q$.

A cursory inspection of the formulas for the coefficients reveals that there is a ubiquitous factor
\[
q+1-q^{1-\delta'} - q^{-\delta}
\]
in each of $A^{\SRW}_{\alpha, d, q}$, $A^{\sph}_{d,q}$, $A^{\ball}_{d,q}$. It turns out that this same interesting factor also appears in the differences $B^{\SRW}_{\alpha, d, q} - B^{\sph}_{d,q}$ and $B^{\sph}_{d,q} - B^{\ball}_{d,q}$, as we will observe in the proof of Theorem \ref{thm:great-ineq}. Because of its ubiquity, we devote one lemma, Lemma \ref{lemma:u-factor}, to proving a simple property of the factor.

In this section, as in the previous ones, we let $d \ge 1$ and $q \ge 2$ be positive integers. We continue to use the notations $\delta' := \left\lceil d/2 \right\rceil$ and $\delta := \left\lfloor d/2 \right\rfloor$. We start by noting the following.

\begin{lemma} \label{lemma:u-factor}
We have $q+1-q^{1-\delta'} - q^{-\delta} \ge q-1 > 0$.
\end{lemma}
\begin{proof}
This follows immediately from the observations that $\delta' \ge 1$ and that $\delta \ge 0$.
\end{proof}

The main result of this section is the following theorem about the relationships between the coefficients $A^{\SRW}_{\alpha, d, q}$, $A^{\sph}_{d,q}$, $A^{\ball}_{d,q}$, $B^{\SRW}_{\alpha, d, q}$, $B^{\sph}_{d,q}$, $B^{\ball}_{d,q}$.

\begin{theorem} \label{thm:great-ineq}
For any $\alpha \in [0,1)$, $d \in \mathbb{Z}_{\ge 1}$, and $q \in \mathbb{Z}_{\ge 2}$, we have the following relations
\begin{align}
& 0 < A^{\SRW}_{\alpha, d, q} < A^{\sph}_{d,q} = A^{\ball}_{d,q} < 2, \label{ineq:A} \\
& \hphantom{0 < } B^{\SRW}_{\alpha, d, q} > B^{\sph}_{d,q} > B^{\ball}_{d,q} \ge \frac{1}{3}. \label{ineq:B}
\end{align}
Furthermore, the equality $B^{\ball}_{d,q} = \frac{1}{3}$ occurs if and only if $(d,q) = (1,2)$.
\end{theorem}

\begin{proof}
It is clear by inspection that $A^{\sph}_{d,q} = A^{\ball}_{d,q} < 2$. To complete the proof of Inequality (\ref{ineq:A}), it suffices to show $0 < A^{\SRW}_{\alpha, d, q} < A^{\sph}_{d,q}$. We proved in Lemma \ref{lemma:u-factor} that the factor $q+1-q^{1-\delta'} - q^{-\delta}$ is a positive number. By canceling the factor, it suffices to show
\[
0 < 2(1-\alpha) \cdot \frac{q-1}{(q+1)^2} < \frac{2}{q+1},
\]
which is clear.

For Inequality (\ref{ineq:B}), we work with one inequality at a time, from the left to the right. For the left one, note that
\begin{align*}
B^{\SRW}_{\alpha, d, q} - B^{\sph}_{d,q} &= \frac{4q}{(q+1)^2(q-1)} \cdot \left( q+1-q^{1-\delta'} - q^{-\delta} \right) \\
&\ge \frac{4q}{(q+1)^2(q-1)} \cdot (q-1) = \frac{4q}{(q+1)^2} > 0.
\end{align*}
For the middle one, note that
\begin{align*}
B^{\sph}_{d,q} - B^{\ball}_{d,q} &= \frac{2}{q^2-1} \cdot \left( q+1-q^{1-\delta'} - q^{-\delta} \right) \\
&\ge \frac{2(q-1)}{q^2-1} = \frac{2}{q+1} > 0.
\end{align*}
For the right one, note that
\[
B^{\ball}_{d,q} > d - \frac{6q+2}{q^2-1}.
\]
The right hand side of the inequality above is at least $1/3$ for any $(d,q) \in \mathbb{Z}_{\ge 1} \times \mathbb{Z}_{\ge 2}$, {\em except} the following thirteen ordered pairs: $(1,2)$, $(1,3)$, $(1,4)$, $(1,5)$, $(1,6)$, $(1,7)$, $(1,8)$, $(1,9)$, $(2,2)$, $(2,3)$, $(2,4)$, $(3,2)$, $(4,2)$. For these exceptional cases, we compute $B^{\ball}_{d,q}$ directly:

\begin{center}
\begin{tabular}{|c|c|}
\hline
&\\[-1em]
$(d,q)$ & $B^{\ball}_{d,q}$ \\
&\\[-1em]
\hline
$(1,2)$ & $1/3$ \\
$(1,3)$ & $1/2$ \\
$(1,4)$ & $3/5$ \\
$(1,5)$ & $2/3$ \\
$(1,6)$ & $5/7$ \\
\hline
\end{tabular}
\hspace{1 cm}
\begin{tabular}{|c|c|}
\hline
&\\[-1em]
$(d,q)$ & $B^{\ball}_{d,q}$ \\
&\\[-1em]
\hline
$(1,7)$ & $3/4$ \\
$(1,8)$ & $7/9$ \\
$(1,9)$ & $4/5$ \\
$(2,2)$ & $2/3$ \\
\hline
\end{tabular}
\hspace{1 cm}
\begin{tabular}{|c|c|}
\hline
&\\[-1em]
$(d,q)$ & $B^{\ball}_{d,q}$ \\
&\\[-1em]
\hline
$(2,3)$ & $1$ \\
$(2,4)$ & $6/5$ \\
$(3,2)$ & $1$ \\
$(4,2)$ & $3/2$ \\
\hline
\end{tabular}
\end{center}

We see that in each of these thirteen cases, the coefficient $B^{\ball}_{d,q}$ is greater than or equal to $1/3$. Note also that the equality case $B^{\ball}_{d,q} = 1/3$ occurs if and only if $(d,q) = (1,2)$. This finishes the proof.
\end{proof}

The lower bounds for $B^{\sph}_{d,q}$ and $B^{\SRW}_{d,q}$ can be determined similarly. We present these in Proposition \ref{prop:min-B} below. Since the proof is similar, we will omit it.

\begin{proposition}\label{prop:min-B}
We have the following.
\begin{itemize}
\item For any $d \in \mathbb{Z}_{\ge 1}$ and $q \in \mathbb{Z}_{\ge 2}$, we have $B^{\sph}_{d,q} \ge 1$, and the equality $B^{\sph}_{d,q} = 1$ occurs if and only if $d = 1$,
\item $\inf \left\{B^{\SRW}_{\alpha, d, q} : \alpha \in [0,1), d \in \mathbb{Z}_{\ge 1}, q \in \mathbb{Z}_{\ge 2} \right\} = 1$,
\item The infimum above is not attained by any values of $\alpha, d, q$.
\end{itemize}
\end{proposition}

\begin{remark}
We remark that one may obtain the property that $A^{\SRW}_{\alpha, d, q}$, $A^{\sph}_{d,q}$, $A^{\ball}_{d,q}$ belong to the interval $[0,2]$ in a different, simpler way. Since these three numbers are coefficients of the main term from the asymptotic formula for a transportation distance, it is clear that they must be non-negative. To obtain that they are at most $2$, note that in each case the distance is bounded above by $2n+d$, because any two masses from $\mu_X^n$ and from $\mu_Y^n$ are at a distance at most $2n+d$ apart. On the other hand, it is not clear how to obtain the {\em ordering} (``SRW-sphere-ball'') of the three coefficients $A^{\SRW}_{\alpha, d, q}$, $A^{\sph}_{d,q}$, $A^{\ball}_{d,q}$, using such a simple argument. It is also unclear how to obtain, from simple arguments, the properties about $B^{\SRW}_{\alpha, d, q}$, $B^{\sph}_{d,q}$, $B^{\ball}_{d,q}$ we proved in Theorem \ref{thm:great-ineq}.
\end{remark}

\bigskip

\section{Graph Statistics from Simple Random Walks}
\label{sec:chi}
In this section we discuss interesting open questions inspired from our analysis of random walks on graphs.

Suppose $\cG = (V,E)$ is a connected, locally-finite, simple graph, with at least two distinct vertices. The definitions of $\fm_X^n$ and $\fm_Y^n$ (which were given for infinite regular trees in Subsection \ref{subsec:LA-SRW}) can be generalized to our graph $\cG$ as follows. Fix a real number $\alpha \in [0,1)$. For each vertex $u \in V(\cG)$, let $\fm_u^n$ be the probability measure obtained from the simple random walk starting at $u$ with laziness $\alpha$ after $n$ steps.

We define the following four graph statistics:
\[
\chi_{\upharpoonleft\!\upharpoonright}(\cG;\alpha) := \sup_{X,Y \in V(\cG)} \limsup_{n \to \infty} \frac{W_1\!\left( \fm_X^n, \fm_Y^n \right)}{n},
\]
\[
\chi_{\upharpoonleft\!\downharpoonright}(\cG;\alpha) := \sup_{X,Y \in V(\cG)} \liminf_{n \to \infty} \frac{W_1\!\left( \fm_X^n, \fm_Y^n \right)}{n},
\]
\[
\chi_{\downharpoonleft\!\upharpoonright}(\cG;\alpha) := \inf_{\substack{X,Y \in V(\cG) \\ X \neq Y}} \limsup_{n \to \infty} \frac{W_1\!\left( \fm_X^n, \fm_Y^n \right)}{n},
\]
and
\[
\chi_{\downharpoonleft\!\downharpoonright}(\cG;\alpha) := \inf_{\substack{X,Y \in V(\cG) \\ X \neq Y}} \liminf_{n \to \infty} \frac{W_1\!\left( \fm_X^n, \fm_Y^n \right)}{n}.
\]
It is not hard to see that these four real numbers belong to the interval $[0,2]$. It is also not hard to see that these four graph statistics are zero when $\cG$ is finite, or when $\cG = \mathbb{Z}^N$ for some positive integer $N$. In this paper, our work has shown the following.

\begin{proposition} \label{prop: four_stat_tree}
For each real number $\alpha \in [0,1)$ and each positive integer $q \ge 2$, we have
\[
\chi_{\downharpoonleft\!\upharpoonright}(\mathbb{T}_{q+1};\alpha) = \chi_{\downharpoonleft\!\downharpoonright}(\mathbb{T}_{q+1};\alpha) = 2(1-\alpha) \left( \frac{q-1}{q+1} \right)^2,
\]
and
\[
\chi_{\upharpoonleft\!\upharpoonright}(\mathbb{T}_{q+1};\alpha) = \chi_{\upharpoonleft\!\downharpoonright}(\mathbb{T}_{q+1};\alpha) = 2(1-\alpha) \left( \frac{q-1}{q+1} \right).
\]
\end{proposition}

It would be interesting to compute these numbers for other cases. Even in the case of half-spaces where $\cG = \mathbb{Z}_{\ge 0} \times \mathbb{Z}^N$, we think that the calculation would be non-trivial. We call this the ``quadrant problem,'' which we phrase as follows.

\begin{problem}[The Quadrant Problem]
Let $M$ and $N$ be non-negative integers. Let $\alpha \in [0,1)$ be a real number. Compute
\[
\chi_{\upharpoonleft\!\upharpoonright}\left(\mathbb{Z}_{\ge 0}^M \times \mathbb{Z}^N ;\alpha\right), \chi_{\upharpoonleft\!\downharpoonright}\left(\mathbb{Z}_{\ge 0}^M \times \mathbb{Z}^N ;\alpha\right), \chi_{\downharpoonleft\!\upharpoonright}\left(\mathbb{Z}_{\ge 0}^M \times \mathbb{Z}^N ;\alpha\right), \chi_{\downharpoonleft\!\downharpoonright}\left(\mathbb{Z}_{\ge 0}^M \times \mathbb{Z}^N ;\alpha\right).
\]
\end{problem}
Note that when $M = 0$, the four statistics are all zero. This is because in $\mathbb{Z}^N$, the distance between $\fm_X^n$ and $\fm_Y^n$ is always $\dist(X,Y) = d$. On the other hand, when $M \ge 1$, the problem is more interesting.

Proposition \ref{prop: four_stat_tree} leads to another question as follows.

\begin{problem}
In Proposition \ref{prop: four_stat_tree}, we see that when $\cG=\mathbb{T}_{q+1}$, we have the equalities $\chi_{\downharpoonleft\!\upharpoonright}(\cG;\alpha) = \chi_{\downharpoonleft\!\downharpoonright}(\cG;\alpha)$ and $\chi_{\upharpoonleft\!\upharpoonright}(\cG;\alpha) = \chi_{\upharpoonleft\!\downharpoonright}(\cG;\alpha)$. It is natural to ask whether this is a general phenomenon. Do these equations hold for any locally finite connected graph $\cG$? 
\end{problem}

We would also be quite interested in the case where $\cG$ is a Cayley graph. Let us pose a general question here. What is an efficient way to compute these four graph statistics in general?

\bigskip

\section{Relation to Coarse Ricci Curvature}
\label{sec:curvature}

In this section, we would like to mention a concept in geometry which motivates our problem of computing the Wasserstein distance: Ollivier's large-scale coarse Ricci curvature. Following the papers by Ollivier \cite{Ollivier}, Lin-Lu-Yau \cite{LLY11}, and Paulin \cite{Paulin16}, we define the \emph{$n$-scale coarse Ricci curvature} for different $X,Y\in V(\cG)$ to be
\[ \kappa_{n,\alpha}(X,Y):=1-\frac{W_1(\fm^n_X,\fm^n_Y)}{\dist(X,Y)}.
\]

Our work in Section \ref{sec:lin-asymp} can be understood as an explicit computation of a coarse Ricci curvature. In the case of $\cG=\bbT_{q+1}$, we deduce from our asymptotic formula in Theorem \ref{thm:SRW_An+B} that
\begin{align*}
	\kappa_{n,\alpha}(X,Y)
	&=\frac{-2n(1-\alpha)}{\dist(X,Y)}(q+1 - q^{1-\delta'} - q^{-\delta}) \cdot \frac{q-1}{(q+1)^2} \\
	&\phantom{=}
	- \frac{1}{\dist(X,Y)} \left(\frac{2(\delta q^{-\delta} + \delta' q^{1-\delta'})}{q+1} + \frac{2(q^{1-\delta} - q^{1-\delta'})}{(q+1)^2}\right) + o(1),
\end{align*}
where we fix $\alpha,X,Y$ and let $n$ go to infinity. Here we recall our notations $\delta:=\lfloor \dist(X,Y)/2\rfloor$ and $\delta':=\lceil \dist(X,Y)/2\rceil$.

In \cite[Example 15]{Ollivier}, Ollivier also mentions in the case of Cayley graphs of hyperbolic groups (which include infinite regular trees) the asymptotic behavior of the $n$-scale coarse Ricci curvature $\kappa_{n,0}(X,Y)$, as $\dist(X,Y)$ and $n$ tend to infinity.

On the other hand, it is also an interesting problem to find an asymptotic behavior or an estimate for the $n$-scale coarse Ricci curvature for other non-hyperbolic groups. In his survey \cite[Problem C]{Ollsurvey}, Ollivier asks whether the Cayley graph of the discrete Heisenberg group $H_3(\bbZ)=\langle\, a,b,c \mid ac=ca,\ bc=cb, aba^{-1}b^{-1}=c\,\rangle$ has $n$-scale coarse Ricci curvature approaching to $1$ as $n$ tends to infinity.

\section*{Acknowledgments}
We firstly thank Nicolas~Juillet whose discussion with us about random walks on the discrete Heisenberg group became the starting point of this project. We are grateful to Wijit~Yangjit for, in many occasions, sharing with us his complex analytic insights. We thank Andrew Wade for recommending Woess' book. We thank Norbert~Peyerimhoff and Shiping~Liu for sharing geometric viewpoints about curvature and transportation distance and for giving comments on an earlier version of the manuscript. We thank Sophia~Benjamin, Arushi~Mantri, Quinn~Perian, and Sorawee~Porncharoenwase for helpful discussions. We used {\em Desmos}, {\em Macaulay2} (on {\em SageMathCell}), {\em python}, {\em R}, and {\em Wolfram Alpha} to help with computations.

\bibliographystyle{alpha}
\bibliography{ref}

\appendix

\bigskip

\section{Calculation of $W_1\!\left( \mu_X^n, \mu_Y^n \right)$ via the Flow}
\label{sec:appendix}
In Theorem \ref{thm:W1=S1+S2+S3}, we have shown a computation of $W_1\!\left( \mu_X^n , \mu_Y^n \right)$ via the potential. For the given $\psi$, we found a good potential $\Phi$ with respect to $\psi$, and used the potential function to calculate the transportation distance. In Section \ref{sec:gentree}, we saw that we can calculate the transportation distance in two ways: (i) via the potential, and (ii) via the flow. Here, we will show a computation via the flow.

We recall the notation $\psi^{\abs}$ from Section \ref{sec:prelim} and the partition of $V(\cG)$ into $V_{i,h}$ from Section \ref{sec:radial}.

In this calculation via the flow, we do not need to assign potentials to the vertices. Instead, we calculate the values of $\psi^{\abs}(e)$ for different edges $e \in E(\cG)$. Let's partition the edges as follows:
\[
E(\cG) := \mathcal{E}_1 \cup \mathcal{E}_2 \cup \mathcal{E}_3,
\]
where the set $\mathcal{E}_1$ contains the edges with at least one endpoint belonging to
\[
\left(\bigcup_{h = 1}^{\infty} V_{0,h}\right) \cup \left(\bigcup_{h = 1}^{\infty} V_{d,h}\right),
\]
the set $\mathcal{E}_2$ contains the edges with at least one endpoint belonging to
\[
\bigcup_{i=1}^{d-1} \bigcup_{h=1}^{\infty} V_{i,h},
\]
and the set $\mathcal{E}_3$ contains the $d$ edges on the path between $X$ and $Y$.

Using the calculation via the flow described in Section \ref{sec:gentree}, we obtain the following formula for the transportation distance $W_1\!\left( \mu_X^n, \mu_Y^n \right)$:
\[
W_1\!\left( \mu_X^n, \mu_Y^n \right) = \sum_{e \in E(\cG)} \psi^{\abs}(e) = \sum_{e \in \mathcal{E}_1} \psi^{\abs}(e) + \sum_{e \in \mathcal{E}_2} \psi^{\abs}(e) + \sum_{e \in \mathcal{E}_3} \psi^{\abs}(e).
\]

To write down the formula explicitly, we compute each sum directly. We find
\[
\sum_{e \in \mathcal{E}_1} \psi^{\abs}(e) = 2 \sum_{h=0}^{\infty} \sum_{i=1}^{\infty} q^{h+i} \left( g(h+i,n) - g(d+h+i,n) \right),
\]
\[
\sum_{e \in \mathcal{E}_2} \psi^{\abs}(e) = 2 \sum_{b=1}^{\left\lfloor (d-1)/2 \right\rfloor} \sum_{h=0}^{\infty} \sum_{i=1}^{\infty} (q-1) q^{h+i-1} \left( g(b+h+i,n) - g(d-b+h+i,n) \right),
\]
and
\begin{align*}
\sum_{e \in \mathcal{E}_3} \psi^{\abs}(e) &= \sum_{a=0}^{d-1} \sum_{h=0}^{\infty} q^h \left( g(h,n) - g(d+h,n) \right) \\
&\hphantom{=} + \sum_{a=0}^{d-1} \sum_{i=1}^a \left( g(i,n) - g(d-i,n) \right) \\
&\hphantom{=} + \sum_{a=0}^{d-1} \sum_{i=1}^a \sum_{h=1}^{\infty} (q-1) q^{h-1} \left( g(i+h,n) - g(d-i+h,n) \right).
\end{align*}

The formula for $W_1 \! \left( \mu_X^n, \mu_Y^n \right)$ here does not look similar to the one we obtained in Theorem \ref{thm:W1=S1+S2+S3}. For instance, here we have a triple sum. However, one can notice without much difficulty that the formula here can be simplified. It is straightforward to check that the formula we obtain here agrees with the formula $W_1 \! \left( \mu_X^n, \mu_Y^n \right) = S_1 + S_2 + S_3$ in Theorem \ref{thm:W1=S1+S2+S3}. We will briefly describe a way to verify the agreement below.

On the right hand side of the equation for $\sum_{e \in \mathcal{E}_3} \psi^{\abs}(e)$ above, let us denote by $T$ the first double sum, by $T'$ the second double sum, and by $T''$ the triple sum. It is not hard to show that
\[
\left( \sum_{e \in \mathcal{E}_1} \psi^{\abs}(e) \right) + T = S_1,
\]
\[
\left( \sum_{e \in \mathcal{E}_2} \psi^{\abs}(e) \right) + T'' = S_3,
\]
and
\[
T' = S_2.
\]
We sum up the three equalities. The sum of the left hand sides becomes the formula via the flow, while the sum $S_1 + S_2 + S_3$ of the right hand sides is the formula, from Theorem \ref{thm:W1=S1+S2+S3}, via the potential.

Thus we have obtained an alternative derivation of the formula for $W_1\!\left( \mu_X^n, \mu_Y^n \right)$, confirming Theorem \ref{thm:W1=S1+S2+S3}.

\bigskip

\section{Asymptotic Analysis of $\gamma(y)$}
\label{sec:gamma}
In Subsection \ref{subsec:LA-SRW}, when we were analyzing the case of simple random walks, we computed the asymptotic formulas for $[y^n] G_1\!(q,y)$ and $[y^n] G(q,y)$. Another important generating function we found was $\gamma(y)$, but we did not need to find a precise asymptotic formula for $[y^n] \gamma(y)$ in order to obtain Theorem \ref{thm:SRW_An+B}: we only needed a general fact from the theory of Markov chains that the $\gamma_i$ terms are small (in the sense of Corollary \ref{cor:gamma-small-good-asymp}). The readers might wonder about the precise asymptotic behavior of $[y^n] \gamma(y)$ as $n \to \infty$. Our goal of this section is to investigate this. We will prove Theorems \ref{thm:gamma-asymp-alpha>0} and \ref{thm:y^2n-gamma-y}, which describe the asymptotic behavior.

The tool we utilize here is a standard method from analytic combinatorics. We refer the readers to Chapter VI of the book of Flajolet and Sedgewick \cite{FS09}. Due to Markov chain periodicity, the asymptotic behavior of $[y^n] \gamma(y)$ is slightly different when $\alpha > 0$ from when $\alpha = 0$. In Subsection \ref{subsec:alpha>0}, we look at the case $\alpha > 0$. We then turn to the case $\alpha = 0$ in Subsection \ref{subsec:alpha=0}. Later, in Subsections \ref{subsec:NE} and \ref{subsec:GBF}, we discuss examples.

\medskip

Let $q \ge 2$ be a positive integer, and let $\alpha \in [0,1)$ be a real number. Recall from Theorem~\ref{thm:SRW-formula} that the generating function $\gamma(y) \in \mathbb{R}[\![y]\!]$ is given by
\[
\gamma(y) = \frac{q}{\left( \frac{q-1}{2} \right) (1-\alpha y) + \sqrt{\Delta}},
\]
where
\[
\Delta := \left( \frac{q+1}{2} \right)^2 (1-\alpha y)^2 - q(1-\alpha)^2 y^2.
\]

The asymptotic behavior of $[y^n]\gamma(y)$ as $n \to \infty$ depends on whether $\alpha$ is zero or not. We therefore distinguish two cases. We treat the case $\alpha > 0$ in Subsection \ref{subsec:alpha>0}, and then treat the case $\alpha = 0$ in Subsection \ref{subsec:alpha=0}.

\medskip

\subsection{Analysis when $\alpha > 0$}\label{subsec:alpha>0}
In this subsection, assume $\alpha > 0$. We locate the singular point of $\gamma(y)$ of the smallest magnitude in the following lemma.

\begin{lemma}\label{l:alpha>0-sing}
Define
\[
\rho := \frac{q+1}{\alpha \cdot (q+1) + (1-\alpha) \cdot 2\sqrt{q}}.
\]
Then, $\rho \in \mathbb{C}$ is the unique singular point of $\gamma(y)$ of the smallest magnitude.
\end{lemma}
\begin{proof}
First, note that we can write
\[
\Delta = \frac{(q+1)^2}{4} \left( 1 - \frac{\alpha (q+1) - (1-\alpha) 2 \sqrt{q}}{q+1} \cdot y \right) \left( 1 - \frac{\alpha (q+1) + (1-\alpha) 2 \sqrt{q}}{q+1} \cdot y \right),
\]
and
\[
\gamma = \frac{q}{\phi_1} = \frac{\ol{\phi}_1}{(y-1) \left( (1-2\alpha)y + 1 \right)},
\]
where
\[
\ol{\phi}_1 = \left( \frac{q-1}{2} \right) (1-\alpha y) - \sqrt{\Delta}.
\]
At this point, there are four candidates for the smallest magnitude singularity: $1$, $(2\alpha-1)^{-1}$ (if $\alpha \neq 1/2$), $\rho$, and $\left( \frac{\alpha (q+1) - (1-\alpha) 2 \sqrt{q}}{q+1} \right)^{-1}$ (if $\alpha \neq \frac{2\sqrt{q}}{(\sqrt{q}+1)^2}$). It is easy to see that the point $1$ is, in fact, a removable singularity. Moreover, if $\alpha < 1/2$, the point $(2\alpha-1)^{-1}$ is removable. On the other hand, if $\alpha > 1/2$, it is straightforward to check that $\rho < 1/(2\alpha - 1)$, and thus $(2\alpha - 1)^{-1}$ is not a singularity of the smallest magnitude.

If $\alpha \neq \frac{2\sqrt{q}}{(\sqrt{q}+1)^2}$, then by the triangle inequality, we see that
\[
\frac{q+1}{\alpha(q+1) + (1-\alpha) \cdot 2\sqrt{q}} < \left| \frac{q+1}{\alpha(q+1) - (1-\alpha) \cdot 2\sqrt{q}} \right|.
\]
Note that the inequality is strict by the assumption that $\alpha \neq 0$. We have shown that $\rho$ is the unique singular point of $\gamma(y)$ with the smallest magnitude.
\end{proof}

Now that Lemma \ref{l:alpha>0-sing} gives us the location of the singularity of the smallest magnitude, the rest is routine computation. We refer the reader to \cite[Chapter VI]{FS09} for details. The singularity analysis gives the following result.

\begin{theorem}\label{thm:gamma-asymp-alpha>0}
For any $\alpha \in (0,1)$ and $q \in \mathbb{Z}_{\ge 2}$, we have the following asymptotic formula
\[
[y^n] \gamma(y) = \frac{1}{\sqrt{4\pi}} \left( \left( \frac{\alpha}{1-\alpha} \right) (q+1) + 2\sqrt{q} \right)^{3/2} \cdot \frac{(q+1)q^{1/4}}{(q-1)^2} \cdot \rho^{-n} \cdot n^{-3/2} + O\!\left(\rho^{-n} n^{-2}\right),
\]
where
\[
\rho = \frac{q+1}{\alpha \cdot (q+1) + (1-\alpha) \cdot 2\sqrt{q}}.
\]
\end{theorem}

\medskip

\subsection{Analysis when $\alpha = 0$}\label{subsec:alpha=0}
The scenario where $\alpha = 0$ is interesting in many aspects. Combinatorially, the simple random walk is non-lazy. Since our graph is bipartite, the Markov chain from the random walk becomes periodic modulo $2$. Analytically, there are more than one singular points on the circle of smallest singularities. Fortunately, the analysis is not too much different from the previous case. In fact, the specialization $\alpha = 0$ seems to simplify the problem.

We remark that this case $\alpha = 0$ is well-studied in algebraic combinatorics. Many sequences on \cite{OEIS} are concerned with the case. See Subsection \ref{subsec:GBF} for details.

When $\alpha = 0$, we have
\[
\Delta = \left( \frac{q+1}{2} \right)^2 - qy^2,
\]
and we may write
\[
\gamma(y) = \frac{q-1-2\sqrt{\Delta}}{2(y^2-1)}.
\]
Let's define a new generating function $\kappa(z) \in \mathbb{R}[\![z]\!]$ as
\[
\kappa(z) := \frac{q-1-\sqrt{(q+1)^2-4qz}}{2(z-1)}.
\]
Note that $\kappa(y^2) = \gamma(y)$ as generating functions in $y$. Hence, if we write $\kappa$ as the series
\[
\kappa(z) = a_0 + a_1 z + a_2 z^2 + \cdots,
\]
then we have
\[
\gamma(y) = a_0 + a_1 y^2 + a_2 y^4 + \cdots.
\]
It is easy to check that $1 \in \mathbb{C}$ is a removable singularity of $\kappa$, and the smallest-magnitude singularity of $\kappa$ is at
\[
\rho := \frac{(q+1)^2}{4q}.
\]
After we have located the singularity closest to zero, the rest is routine computation. Once again, we refer to \cite[Chapter VI]{FS09} for details. We have
\[
[z^n] \kappa(z) = \frac{1}{\sqrt{\pi}} \cdot \frac{q(q+1)}{(q-1)^2} \cdot \left( \frac{4q}{(q+1)^2} \right)^n n^{-3/2} + O\!\left( \left( \frac{4q}{(q+1)^2} \right)^n n^{-2} \right).
\]
Thus we have proved the following.

\begin{theorem}\label{thm:y^2n-gamma-y}
Let $q \ge 2$ be a positive integer. In the case $\alpha = 0$, we have that
\[
[y^{2n}] \gamma(y) = \frac{1}{\sqrt{\pi}} \cdot \frac{q(q+1)}{(q-1)^2} \cdot \left( \frac{4q}{(q+1)^2} \right)^n n^{-3/2} + O\!\left( \left( \frac{4q}{(q+1)^2} \right)^n n^{-2} \right),
\]
as $n \to \infty$, and $[y^{2n+1}] \gamma(y) = 0$, for $n \ge 0$.
\end{theorem}

\medskip

\subsection{First Example: Noncommutative Expansion}\label{subsec:NE}
A fine example of application of our analysis here is the OEIS sequence A328494 \cite{OEIS}, which contains
\[
f(n) := \text{the constant term of } \left(1+x+y+x^{-1}+y^{-1}\right)^n,
\]
for each $n \in \mathbb{Z}_{\ge 0}$, where $x$ and $y$ are {\em noncommutative} variables. The first few terms are $f(0) = 1$, $f(1) = 1$, $f(2) = 5$, $f(3) = 13$, and $f(4) = 53$ \cite[A328494]{OEIS}. The number $f(n)$ has the following combinatorial description of walks on graph. Suppose we start (at time $t = 0$) at a particular vertex in the infinite regular tree $\mathbb{T}_4$. At each time step, we either stay at the same vertex or move one step to a neighbor. Then, $f(n)$ is the number of different walks we can take from time $t = 0$ to time $t = n$ to come back to the initial vertex at the end of our walk. This is indeed closely related to the generating function $\gamma(y)$ in the case where $\alpha = 1/5$ and $q = 3$: note that
\[
f(n) = 5^n \cdot \left( [y^n] \gamma(y) \right) = [y^n] \gamma(5y).
\]
Using the formula from Theorem \ref{thm:SRW-formula}, we find
\[
\sum_{n=0}^{\infty} f(n) \cdot y^n = \gamma(5y) = \frac{3}{1-y+2\sqrt{1-2y-11y^2}},
\]
which gives the (ordinary) generating function for A328494 \cite{OEIS}.

Now, Theorem \ref{thm:gamma-asymp-alpha>0} gives
\[
[y^n] \gamma(y) = \sqrt{\frac{90+37\sqrt{3}}{4\pi}} \left( \frac{1+2\sqrt{3}}{5} \right)^n n^{-3/2} + O \!\left( \left( \frac{1+2\sqrt{3}}{5} \right)^n n^{-2} \right).
\]
In particular, this yields the asymptotic formula for the terms in A328494:
\[
f(n) \sim \sqrt{\frac{90+37\sqrt{3}}{4\pi}} \cdot (1+2\sqrt{3})^n \cdot n^{-3/2},
\]
as $n \to \infty$.

The example above is easy to generalize to a higher number of noncommutative variables. For instance, we can consider
\[
g(n) := \text{the constant term of } \left(1+x+y+z+x^{-1}+y^{-1}+z^{-1}\right)^n,
\]
for each $n \in \mathbb{Z}_{\ge 0}$, where $x, y, z$ are noncommutative variables. The first few terms of the sequence $\{g(n)\}_{n=0}^{\infty}$ are $g(0) = 1$, $g(1) = 1$, $g(2) = 7$, $g(3) = 19$, $g(4) = 103$, $g(5) = 391$, and $g(6) = 1957$. After applying an analogous analysis as above, we obtain the generating function
\[
\sum_{n=0}^{\infty} g(n) \cdot y^n = \frac{5}{2(1-y) + 3 \sqrt{1-2y-19y^2}},
\]
and the asymptotic formula
\[
g(n) \sim \frac{3}{16} \cdot \sqrt{\frac{230 + 61 \sqrt{5}}{\pi}} \cdot (1+2\sqrt{5})^n \cdot n^{-3/2},
\]
as $n \to \infty$.

\medskip

\subsection{Second Example: Formulas of Boddington's and Kotesovec's}\label{subsec:GBF}
As one might have expected, the case where there is no laziness ($\alpha = 0$) is well-studied in combinatorics. Many sequences on the OEIS \cite{OEIS} are related to this situation, and we give some examples of them below in this subsection. A related question was also considered, for example, in the work of Haiman \cite{Hai93}.

For any $q \in \mathbb{Z}_{\ge 2}$, when $\alpha = 0$, if we let
\[
f(n) = (q+1)^{2n} \left( [y^{2n}] \gamma(y) \right) = [y^{2n}] \gamma((q+1)y),
\]
then $f(n)$ enumerates the number of $2n$-step walks in $\mathbb{T}_{q+1}$ which start and finish at a distinguished point. Our Theorem \ref{thm:SRW-formula} can be seen as a generalization of this object into which laziness is introduced.

When $\alpha = 0$, we have seen in Subsection \ref{subsec:alpha=0} that
\[
\gamma(y) = \frac{q-1-\sqrt{(q+1)^2-4qy^2}}{2(y^2-1)} = \frac{2q}{(q-1)+\sqrt{(q+1)^2-4qy^2}}.
\]
This shows that
\[
\sum_{n=0}^{\infty} f(n) x^n = \frac{2q}{(q-1)+(q+1)\sqrt{1-4qx}}. \tag{$\heartsuit$}
\]
This formula ($\heartsuit$) was discovered by Paul Boddington (cf. A035610 on \cite{OEIS}). (Note that Boddington uses the variable $m$ instead of $q$. The two variables are simply related by $m = q+1$.) Our formula for $\gamma(y)$ in Theorem \ref{thm:SRW-formula} can be therefore called a {\em generalized Boddington's formula}: the formula has an additional laziness parameter $\alpha$.

Many OEIS sequences are related to the formula ($\heartsuit$) of Boddington's. For example, when $q = 2$, we obtain the generating function
\[
\frac{4}{1+3\sqrt{1-8x}},
\]
which is A089022 \cite{OEIS}. On the OEIS page for this sequence, Vaclav Kotesovec computed the asymptotic formula for the terms. When $q = 3$, we find A035610 \cite{OEIS}, where Kotesovec also computed the asymptotic formula for the terms. When $q = 4$, we find A130976 \cite{OEIS}, and Kotesovec also provided the asymptotic formula as well. Our work in Subsection \ref{subsec:alpha=0} produces the asymptotic formula for a general $q \in \mathbb{Z}_{\ge 2}$. Since $f(n) = (q+1)^{2n} \left( [y^{2n}] \gamma(y) \right)$, Theorem \ref{thm:y^2n-gamma-y} yields the general formula
\[
f(n) \sim \frac{q(q+1)}{\sqrt{\pi} (q-1)^2} \cdot (4q)^n n^{-3/2},
\]
as $n \to \infty$. Specializations of this asymptotic formula indeed agree with many aforementioned formulas of Kotesovec's.

\end{document}